\documentclass[reqno, 11pt]{amsart}
\usepackage{amssymb,amsmath,amsthm}

\numberwithin{equation}{section}

\newtheorem{thm}{Theorem}
\newtheorem*{cor}{Corollary}
\newtheorem{lem}{Lemma}

\newcommand{\RR}{\mathcal{R}}
\newcommand{\M}{\mathbf{M}}
\newcommand{\A}{\mathbf{A}}

\newcommand{\R}{\mathbb{R}}

\newcommand{\N}{\mathbb{N}}
\newcommand{\Z}{\mathbb{Z}}

\newcommand{\bd}{\mathbf{d}}

\newcommand{\y}{\mathbf{y}}
\newcommand{\z}{\mathbf{z}}

\renewcommand{\v}{\mathbf{v}}

\newcommand{\mnu}{\boldsymbol{\nu}}
\newcommand{\mmu}{\boldsymbol{\mu}}
\newcommand{\mal}{\boldsymbol{\lambda}}

\newcommand{\x}{{\bf x}}

\newcommand{\bfP}{\mathbb{P}}
\newcommand{\bfA}{\mathbb{A}}

\newcommand{\ma}{\mathbf}
\newcommand{\ben}{\begin{enumerate}}
\newcommand{\een}{\end{enumerate}}
\newcommand{\beq}{\begin{equation}}
\newcommand{\eeq}{\end{equation}}

\newcommand{\ve}{\varepsilon}

\newcommand{\del}{\delta}
\newcommand{\mcal}{\mathcal}
\newcommand{\lab}{\label}
\newcommand{\al}{\alpha}

\newcommand{\la}{\lambda}
\newcommand{\sfl}{\mathsf{\Lambda}}

\newcommand{\sfg}{\mathsf{\Gamma}}

\newcommand{\colt}[2]{\genfrac{}{}{0pt}{1}{#1}{#2}}
\newcommand{\tolt}[3]{\colt{#1}{\colt{#2}{#3}}}
\newcommand{\eqm}[2]{\equiv #1 \pmod{#2}}

\renewcommand{\d}{\mathrm{d}}
\renewcommand{\leq}{\leqslant}
\renewcommand{\geq}{\geqslant}
\renewcommand{\le}{\leqslant}
\renewcommand{\ge}{\geqslant}

\theoremstyle{definition}
\newtheorem*{ack}{Acknowledgements}
\newtheorem*{notat}{Notation}

\renewcommand{\SS}{\mathcal{S}}

\DeclareMathOperator{\hcf}{gcd}
\DeclareMathOperator{\meas}{meas}

\begin{document}

\title{Binary linear forms as sums of two squares}

\author{R. de la Bret\`eche}
\author{T.D. Browning}

\address{
Institut de Math\'ematiques de Jussieu,
Universit\'e Paris 7 Denis Diderot,
Case Postale 7012,
2, Place Jussieu,
F-75251 Paris cedex 05}
\email{breteche@math.jussieu.fr}

\address{School of Mathematics,
University of Bristol, Bristol BS8 1TW}
\email{t.d.browning@bristol.ac.uk}

\date{\today}

\begin{abstract}
We revisit recent work of Heath-Brown on the average
order of the quantity $r(L_1(\x))\cdots r(L_4(\x))$, 
for  suitable binary linear forms $L_1,\ldots, L_4$, 
as $\x=(x_1,x_2)$ ranges over quite general regions in $\Z^2$. In
addition to improving the error term in Heath-Brown's estimate
we generalise his result to cover a wider class of linear forms.
\end{abstract}

\maketitle

\section{Introduction}

Let $L_1,\ldots,L_4\in\Z[x_1,x_2]$ be binary linear forms,
and let $\mcal{R}\subset \R^2$ be any bounded region.
This paper is motivated  by the question of determining conditions on
$L_1,\ldots,L_4$ and $\mcal{R}$ under which it is possible to establish
an asymptotic formula for the sum
$$
  S(X):=\sum_{\x=(x_1,x_2)\in \Z^2\cap X\mcal{R}}r(L_1(\x))r(L_2(\x))r(L_3(\x)) r(L_4(\x)),
$$
as $X\rightarrow \infty$, where $X\mcal{R}:= \{X\x: \x\in \mcal{R}\}.$  
The problem of determining an upper bound for $S(X)$ is substantially
easier. In fact the main result in the authors' recent investigation 
\cite{nair}
into the average order of arithmetic functions over the values of binary forms can
easily be used to show that
$
S(X)\ll X^2,
$
provided that no two of $L_1,\ldots,L_4$ are proportional. 
In trying to establish an asymptotic formula for $S(X)$ there is no
real loss in generality in restricting ones attention to
the corresponding sum in which one of the variables $x_1,x_2$ is
odd. For $j \in \{*,0,1\}$, let us write $S_{j}(X)$ for the 
corresponding 
sum in which $x_1$ is odd and $x_2 \equiv j \bmod{2}$, where the case
$j=*$ means that no $2$-adic restriction is placed on $x_2$.

Our point of departure is recent work of Heath-Brown \cite{h-b03},
which establishes an asymptotic formula for $S_*(X)$ when
$L_1,\ldots,L_4$ and $\mcal{R}$ satisfy the following normalisation
hypothesis:
\begin{enumerate}
\item[(i)]
$\mcal{R}$ is an open, bounded and convex region, with a piecewise continuously
differentiable boundary,
\item[(ii)]
no two of $L_1,\ldots,L_4$ are proportional,
\item[(iii)]
$L_i(\x)>0$ for all $\x
\in \mcal{R}$,
\item[(iv)]
we have $L_i(\x)\equiv x_1 \bmod{4}$.
\end{enumerate}
Here, as throughout our work, the index $i$ denotes a generic
element of the set $\{1,2,3,4\}$.
We will henceforth say that $L_1,\ldots,L_4, \mcal{R}$ ``satisfy  \textsf{NH}$_0$'' if these
four conditions hold.
The first three conditions are all quite natural, and don't impose any
serious constraint on $L_1,\ldots,L_4, \mcal{R}$. The fourth condition is
more problematic however, especially when it comes to applying the
result in other contexts.
We will return to this issue shortly. For the moment we
concern ourselves with presenting a refinement of Heath-Brown's
result.  It will be necessary to begin by introducing some more 
notation. 

For given $L_1,\ldots,L_4,\mcal{R}$ we 
will write 
\begin{equation}
  \label{eq:Linf}
L_\infty=L_\infty(L_1,\ldots,L_4):=\max_{1\leq i\leq 4}\|L_i\|,
\end{equation}
where $\|L_i\|$ denotes the maximum
modulus of the coefficients of $L_i$, and
\begin{equation}
  \label{eq:rinf}
r_\infty=r_\infty(\mcal{R}):=\sup_{\x\in\mcal{R}}\max\{|x_1|,|x_2|\}.
\end{equation}
Furthermore, let
\begin{equation}
  \label{eq:r'}
r'=r'(L_1,\ldots,L_4,\mcal{R}):=\sup_{\x\in\mcal{R}}\max_{1\leq i \leq
4}|L_i(\ma{x})|.
\end{equation} 
Define the real number
\begin{equation}\label{defeta} 
\eta:=1-\frac{1+\log\log 2}{\log 2},
\end{equation}
with numerical value $0.08607\ldots$, and let $\chi$ be the
non-principal character modulo $4$ defined multiplicatively by
$$
\chi(p):=\left\{
\begin{array}{ll}
+1, & \mbox{if $p\equiv 1\bmod{4}$},\\
-1, & \mbox{if $p\equiv 3\bmod{4}$},\\
0, & \mbox{if $p=2$}.
\end{array}
\right.
$$
We are now ready to reveal our first result. 

\begin{thm}\lab{main0}
Assume that $L_1,\ldots,L_4,\mcal{R}$ satisfy \textsf{NH}$_0$, and 
let $\varepsilon>0$.  Suppose that $r'X^{1-\ve}\geq 1$. 
Then we have
$$
S_{*}(X)
= 4\pi^4 \meas(\mcal{R})X^2 \prod_{p>2}\sigma_p^* +
O\Big(\frac{L_\infty^{ \varepsilon}r_\infty r'X^2}{(\log X)^{ \eta-\varepsilon}}\Big),
$$
where 
\begin{equation}
   \label{eq:sig*}
\sigma_p^*:=\Big(1-\frac{\chi(p)}{p}\Big)^4
\sum_{a,b,c,d=0}^{\infty}\chi(p)^{a+b+c+d}\rho_*(p^a,p^b,p^c,p^d)^{-1},
\end{equation}
and
\begin{equation}
  \label{eq:rho*}
\rho_*(\ma{h}):=\det\{\x\in\Z^2: h_i\mid L_i(\x)\}
\end{equation}
as a sublattice of $\Z^2$.
Moreover, the product $\prod \sigma_p^*$ is absolutely convergent.
\end{thm}

The implied constant in this estimate is allowed to depend upon the
choice of $\varepsilon$, a convention that we will adopt 
for all of the implied constants in this paper.
It would be straightforward to replace the term $(\log X)^\ve$ by 
$(\log\log X)^A$ in the error term, for some explicit value of $A$.  
For the purposes of comparison, we note that \cite[Theorem 1]{h-b03} 
consists of an asymptotic formula for $S_*(X)$ with error
$$
O_{L_1\ldots,L_4,\mcal{R}}\Big( \frac{X^2(\log\log
X)^{15/4}}{(\log X)^{\eta/2}}\Big). 
$$ 
Here there is an unspecified dependence on $L_1,\ldots,L_4,\mcal{R}$, 
and  $\eta$ is given by \eqref{defeta}.  
Thus Theorem \ref{main0} is stronger than  \cite[Theorem 1]{h-b03}
in two essential aspects. Firstly, we have been able to obtain complete
uniformity in $L_1,\ldots,L_4,\mcal{R}$ in the error term, and
secondly, our exponent of $\log X$ is almost twice the size.

Our next result extends Theorem \ref{main0} to 
points running over vectors belonging to suitable sublattices of
$\Z^2$. The advantages of such a generalisation will be made clear 
shortly.  
For any $\ma{D}=(D_1,\ldots,D_4) \in \N^4$, we let
\begin{equation}
   \label{eq:lattice}
\sfg_{\ma{D}}=\sfg(\ma{D};L_1,\ldots,L_4):= \{\x\in\Z^2: D_i\mid L_i(\x)\}.
\end{equation}
Then $\sfg_{\ma{D}}\subseteq \Z^2$ is an integer lattice of
rank $2$. Next, let $\ma{d}=(d_1,\ldots,d_4)\in\N^4$ and assume that $d_i\mid D_i$.  In particular it follows that
$
\sfg_{\ma{D}}\subseteq \sfg_{\ma{d}}.
$
Throughout this paper we will focus our attention on $(\ma{d},\ma{D})
\in \mcal{D}$, where
\begin{equation}
   \label{eq:D}
\mcal{D}:= \big\{
(\ma{d}, \ma{D})\in \N^8: 2\nmid d_iD_i,  ~d_i\mid D_i
\big\}.
\end{equation}
For $j \in \{*,0,1\}$ the goal is to establish
an asymptotic formula for
\begin{equation}
   \label{eq:Sj}
S_{j}(X;\ma{d},\sfg_{\ma{D}}):=\sum_{\tolt{\x\in \sfg_{\ma{D}}\cap
     X\mcal{R}}{2\nmid x_1}{x_2\equiv j \bmod{2}}}
r\Big(\frac{L_1(\x)}{d_1}\Big)
r\Big(\frac{L_2(\x)}{d_2}\Big)r\Big(\frac{L_3(\x)}{d_3}\Big)
r\Big(\frac{L_4(\x)}{d_4}\Big).
\end{equation}
It is clear that
$S_{j}(X)=S_j(X;(1,1,1,1),\Z^2)$ for each $j \in \{*,0,1\}$, in the
above notation.

For given $\ma{d}\in \N^4$ with odd components, let us 
say that $L_1,\ldots,L_4,\mcal{R}$ ``satisfy \textsf{NH}$_0(\ma{d})$''
if they satisfy the conditions in \textsf{NH}$_0$, but with (iv) replaced
by 
\begin{enumerate} 
\item[(iv)$_{\ma{d}}$]
we have $L_i(\x)\equiv d_i x_1 \bmod{4}$.
\end{enumerate}
When $d_i\equiv 1 \bmod 4$ for each $i$, it is clear that
(iv)$_\ma{d}$ coincides with (iv).
Let $[a,b]$ denote the least common multiple of any two positive
integers $a,b$. 
The results that we obtain involve the quantity
\begin{equation}
   \label{eq:rho0}
\rho_0(\ma{h})
:=\frac{\det
 \sfg\big(([D_1,d_1h_1],\ldots,[D_4,d_1h_4]);L_1,\ldots,L_4\big)}{\det \sfg(\ma{D};L_1,\ldots,L_4)},
  \end{equation}
which we will occasionally denote by $\rho_0(\ma{h};\ma{D};L_1,\ldots,L_4)$.
Specifically we have local factors
\begin{equation}
   \label{eq:sig}
\sigma_p:=\Big(1-\frac{\chi(p)}{p}\Big)^4
\sum_{a,b,c,d=0}^{\infty}\chi(p)^{a+b+c+d}\rho_0(p^a,p^b,p^c,p^d)^{-1},
\end{equation}
defined for any prime $p>2$. In view of \eqref{eq:sig*} and \eqref{eq:rho*}, we note that
$\rho_0(\ma{h})=\rho_*(\ma{h})$ and $\sigma_p=\sigma_p^*$ when $D_i=1$, 
since then $\sfg_{\ma{D}}=\Z^2$. 
Bearing all this notation in mind, we have the following result.

\begin{thm}\lab{main1}
Let $(\ma{d},\ma{D})
\in \mcal{D}$ and assume that $L_1,\ldots,L_4,\mcal{R}$ satisfy \textsf{NH}$_0(\ma{d})$.
Let $\varepsilon>0$ and suppose that $r'X^{1-\ve}\geq 1$. 
Let $j \in \{*,0,1\}$. Then we have
$$
S_{j}(X;\ma{d},\sfg_{\ma{D}})
= \frac{\delta_j\pi^4 \meas(\mcal{R})}{\det \sfg_{\ma{D}}}X^2 \prod_{p>2}\sigma_p +
 O\Big(\frac{D^\ve L_\infty^{ \varepsilon}r_\infty r'X^2}{(\log X)^{ \eta-\varepsilon}}\Big),
$$
where $D:=D_1D_2D_3D_4$ and 
\begin{equation}
   \label{eq:dj}
\delta_j:=\left\{
\begin{array}{ll}
2, & \mbox{if $j=0,1$},\\
4, & \mbox{if $j=*$},
\end{array}
\right.
\end{equation}
and $L_\infty, r_\infty, r'$ 
 are given by \eqref{eq:Linf},
\eqref{eq:rinf} and \eqref{eq:r'},  respectively.
Moreover, the product $\prod \sigma_p$ is absolutely convergent.
\end{thm}

Taking $d_i=D_i=1$ and $j=*$ in the statement of Theorem~\ref{main1},
so that in particular $\sfg_{\ma{D}}=\Z^2$, we retrieve Theorem \ref{main0}.
In fact Theorem \ref{main1} is a rather routine deduction from
Theorem~\ref{main0}. This will be carried out in \S \ref{lattices}.

We now return to the normalisation conditions (i)--(iv)$_\ma{d}$ that form the
basis of  Theorem \ref{main1}. As indicated above, one of the main  
motivations behind writing this paper has been to weaken these conditions
somewhat. In fact we will be able to replace condition (iv)$_{\ma{d}}$ by either of 
\begin{enumerate} 
\item[(iv$'$)$_{\ma{d}}$]
the coefficients of $L_3,L_4$ are all non-zero and 
there exist integers $k_1,k_2\geq 0$ such that
$$
2^{-k_1}L_1(\x)\equiv d_1x_1 \pmod{4}, \quad 2^{-k_2}L_2(\x)\equiv d_2x_1 \pmod{4},
$$
\end{enumerate}
or
\begin{enumerate}
\item[(iv$''$)$_{\ma{d}}$]
the coefficients of $L_3,L_4$ are all non-zero and 
there exist integers $k_1,k_2\geq 0$ such that
$$
2^{-k_1}L_1(\x)\equiv d_1x_1 \pmod{4}, \quad 2^{-k_2}L_2(\x)\equiv x_2 \pmod{4}.
$$
\end{enumerate}
Accordingly, we will say that $L_1,\ldots,L_4, \mcal{R}$ ``satisfy
\textsf{NH}$_1(\ma{d})$'' if they satisfy conditions (i)--(iii) and (iv$'$)$_{\ma{d}}$,
and we will say that  $L_1,\ldots,L_4, \mcal{R}$ ``satisfy
\textsf{NH}$_2(\ma{d})$'' if 
together with (i)--(iii), they satisfy condition~(iv$''$)$_{\ma{d}}$. 
The condition that  none of the coefficients of $L_3,L_4$ are zero is
equivalent to the statement that neither $L_3$ nor $L_4$ is
proportional to $x_1$ or $x_2$. Condition (ii) ensures that no two of $L_1,\ldots,L_4$
are proportional, and so if $L_3$ or $L_4$ is proportional to one of
$x_1$ or $x_2$, then there are at least two forms among $L_1,\ldots,L_4$
that are not proportional to $x_1$ or $x_2$. After a possible relabeling, 
therefore, one may always assume that the coefficients 
of $L_3,L_4$ are non-zero.

The asymptotic formula that we obtain under these new 
hypotheses is more complicated than  Theorem \ref{main1}, and
intimately depends on the coefficients of $L_3, L_4$.
Suppose that
\begin{equation}\label{L3L4}
L_3(\x)=a_3x_1+b_3x_2,\quad L_4(\x)=a_4x_1+b_4x_2,
\end{equation}
and write
$$
\A=\Big( \begin{array}{cc}
a_3&b_3\\
a_4 &b_4
\end{array}
\Big),
$$
for the associated matrix.  In particular for $L_1,\ldots,L_4$
satisfying any of the normalisation conditions above, we may assume
that $\A$ is an integer valued matrix with non-zero determinant and
non-zero entries.

Let $(j,k)\in \{*,0,1\}\times\{0,1,2\}$. 
We proceed to introduce a quantity
$\delta_{j,k}(\A,\ma{d})\in \R$, which 
will correspond to the $2$-adic density of vectors $\x\in \Z^2$ with
$x_1\equiv 1 \bmod 4$  
and $x_2 \equiv j \bmod{2}$, for which the
corresponding summand in \eqref{eq:Sj} is non-zero for $L_1,\ldots,L_4,\mcal{R}$ 
satisfying \textsf{NH}$_k(\ma{d})$. 
Let 
\begin{equation} 
  \label{eq:Ar}
E_n:=\{x\in  \Z/2^n\Z : ~\exists ~\nu\in\Z_{\geq 0}, ~
  2^{-\nu}x\equiv 1 \bmod{4}\}, 
\end{equation}
for any $n \in \N$. Then we may set 
\begin{equation}\label{eq:sig2}
\delta_{j,k}(\A,\ma{d}):=\lim_{n\to\infty} \frac{1}{2^{2n-4}} \#\left\{
\x\in(\Z/2^n\Z)^2 :
\begin{array}{l}
x_1\equiv 1 \bmod 4\\ 
x_2\equiv j \bmod{2}\\ 
L_i(\x)\in d_iE_n 
\end{array}\right\}.
\end{equation}
This limit plainly always exists and is contained in the interval $[0,4]$.
It will ease notation if we simply write $\delta_{j,k}(\A)$ for
$\delta_{j,k}(\A,\ma{d})$ in all that follows. We will calculate
this quantity explicitly in \S \ref{2-adic}. We are 
now ready to reveal our main result. 

\begin{thm}\lab{main2} 
Let $(\ma{d},\ma{D}) \in \mcal{D}$ and assume that $L_1,\ldots,L_4,\mcal{R}$
satisfy \textsf{NH}$_k(\ma{d})$ 
 for $k\in \{0,1,2\}$. 
Let $\varepsilon>0$ and suppose that $r'X^{1-\ve}\geq 1$. 
Let $j \in \{*,0,1\}$. Then we have
$$
S_j(X;\ma{d},\sfg_{\ma{D}})
=cX^2 + O\Big(\frac{D^\ve L_\infty^{ \varepsilon}r_\infty r'X^2}{(\log X)^{ \eta-\varepsilon}}\Big),
$$
where 
$$
c=
\delta_{j,k}(\A)\frac{\pi^4 \meas(\mcal{R})}{\det \sfg_{\ma{D}}} \prod_{p>2}\sigma_p.
$$
\end{thm}

It is rather trivial to check that $\delta_{j,0}(\A)=\delta_j$, 
in the notation of \eqref{eq:dj}. 
Hence the statement of Theorem~\ref{main2} reduces to Theorem
\ref{main1} when $k=0$.  The proof of Theorem \ref{main2} for $k=1,2$ uses Theorem~\ref{main1} as a
crucial ingredient, but it will be significantly more complicated than
the corresponding deduction of Theorem~\ref{main1} from Theorem
\ref{main0}. This will be carried out in \S \ref{s:t2}.
The underlying idea is to find appropriate linear transformations
that take the relevant linear forms into forms that satisfy the
normalisation conditions (i)--(iv)$_{\ma{d}}$, thereby
bringing the problem in line for an application of Theorem~\ref{main1}.
In practice the choice of transformation depends closely upon the 
coefficients of $L_3, L_4$, and a careful case by case analysis is necessary
to deal with all eventualities.

While interesting in its own right, the study of sums like
\eqref{eq:Sj} is intimately related to problems involving the 
distribution of integer and rational points on algebraic varieties. 
In fact estimating $S_j(X;\ma{d},\sfg_{\ma{D}})$ boils 
down to counting integer points on the affine variety
\begin{equation}
  \label{eq:torsor}
L_i(x_1,x_2)=d_i(s_i^2+t_i^2), \quad (1\leq i\leq 4),
\end{equation}
in $\bfA^{10}$, with $x_1,x_2$ restricted in some way.
Viewed in this light it might be expected that the constant $c$ in Theorem
\ref{main2} admits an interpretation as a product of local
densities. Our next goal is to show that this is indeed the case.

Let $\mal=(\lambda_1,\ldots, \lambda_4)\in\Z_{\geq 0}^4$
and let $\mmu=(\mu_1,\ldots, \mu_4)\in\Z_{\geq 0}^4$. 
Given any prime $p>2$, let
$$
N_{\mal,\mmu}(p^n):=\#\Big\{(\x,\ma{s},\ma{t})\in
(\Z/p^n\Z)^{10}: \begin{array}{l}
L_i(x_1,x_2)\equiv p^{\lambda_i}(s_i^2+t_i^2) \bmod{p^n}\\
p^{\mu_i} \mid L_i(x_1,x_2)
\end{array}
\Big\},
$$
and define 
\begin{equation}
  \label{eq:def-sigp}
\omega_{\mal,\mmu}(p):=\lim_{n\rightarrow \infty}p^{-6n-\la_1-\cdots-\la_4}N_{\mal,\mmu}(p^n).
\end{equation}
This corresponds to the $p$-adic density on a variety of the
form \eqref{eq:torsor}, in which the points are restricted to lie on a
certain sublattice of $\Z/p^n\Z$.

Turning to the case $p=2$, let
$$
N_{j,k,\bd}(2^n):=\#\Big\{(\x,\ma{s},\ma{t})\in
(\Z/2^n\Z)^{10}: \begin{array}{l}
L_i(x_1,x_2)\equiv d_i(s_i^2+t_i^2) \bmod{2^n}\\
x_1 \equiv 1 \bmod{4}, ~x_{2}\equiv j \bmod{2}
\end{array}
\Big\},
$$
for any $(j,k) \in \{*,0,1\}\times\{0,1,2\}$  and any $\bd\in\N^4$ such that $2\nmid
d_1\cdots d_4$.  Here the subscript $k$ indicates that 
$L_1,\ldots,L_4,\mcal{R}$ are assumed to satisfy \textsf{NH}$_k(\ma{d})$.
The corresponding $2$-adic density is given by 
\begin{equation}
  \label{eq:eq:def-dig2}
\omega_{j,k,\bd}(2):=\lim_{n\rightarrow \infty}2^{-6n}N_{j,k,\bd}(2^n).
\end{equation}
Finally, we let $\omega_{\mcal{R}}(\infty)$ denote the archimedean density of
solutions  to the system of equations
\eqref{eq:torsor}, for which $(\x,\ma{s},\ma{t})\in\mcal{R}\times
\R^{8}$. 
We will establish the 
following result in \S \ref{s:local}.

\begin{thm}\lab{main3} 
We have 
$$
c=\omega_{\mcal{R}}(\infty)
\omega_{j,k,\ma{d}}(2)\prod_{p>2}\omega_{\mal,\mmu}(p),
$$
in the statement of Theorem \ref{main2}, with
\begin{align*}
\mal=\big(\nu_p(d_{1}),\ldots,\nu_p(d_{4})\big),
\quad \mmu=\big(\nu_p(D_{1}),\ldots,\nu_p(D_{4})\big).
\end{align*}
\end{thm}

It turns out that the system of equations in \eqref{eq:torsor} play 
the role of descent varieties for the pair of equations
$$
L_1(x_1,x_2)L_2(x_1,x_2)=x_3^2+x_4^2, \quad 
L_3(x_1,x_2)L_4(x_1,x_2)=x_5^2+x_6^2,
$$
for binary linear forms $L_1,\ldots,L_4$ defined over $\Z$. 
This defines a geometrically integral threefold $V\subset \bfP^5$, and it is
natural to try and estimate the number $N(X)$ of rational points 
on $V$ with height at most $X$, as $X\rightarrow \infty.$  In fact there is a
very precise conjecture due to Manin \cite{fmt} which relates the
growth of $N(X)$ to the intrinsic geometry of $V$. It is easily
checked that $V$ is a singular
variety with finite singular locus consisting of double points. If $\widetilde{V}$ denotes the minimal
desingularisation of $V$, then the Picard group of $\widetilde{V}$ has
rank $1$. Moreover, $K_{\widetilde{V}}+2H$ is effective,
where $K_{\widetilde{V}}$ is a canonical divisor and $H$ is a hyperplane
section. Thus Manin's conjecture predicts the asymptotic behaviour 
$
N(X)=c_V
X^2(1+o(1)),
$
as $X\rightarrow \infty$, for a suitable constant $c_V\geq 0$.

Building on his 
investigation \cite[Theorem 1]{h-b03} into the sum $S_*(X)$ defined above, Heath-Brown provides
considerable evidence for this conjecture when
$L_1,\ldots,L_4, \mcal{R}$ satisfy a certain normalisation hypothesis, which he
labels \textbf{NC2}. This coincides with the conditions (i)--(iii) in
\textsf{NH}$_0$, but with (iv) replaced by the condition that
$$
L_1(\x)\equiv L_2(\x) \equiv \nu x_1 \pmod{4},
\quad
L_3(\x)\equiv L_4(\x) \equiv \nu' x_1 \pmod{4},
$$
for appropriate $\nu,\nu' =\pm 1.$
The outcome of Heath-Brown's investigation is \cite[Theorem
2]{h-b03}. Under \textbf{NC2} this establishes the existence of a constant $c\geq 0$ and a function
$E(X)=o(X^2)$, such that
\begin{equation}
  \label{eq:bell}
\sum_{\colt{\x \in \Z^2\cap X \mcal{R}}{x_1\equiv 1\bmod{2}}} 
r(L_1(\x)L_2(\x))r(L_3(\x)L_4(\x))=cX^2+O(E(X)).
\end{equation}
The explicit value of $c$ is rather complicated to state and will not
be given here. One of the features of Heath-Brown's proof is that it doesn't
easily lead to an explicit error function $E(X)$.
An examination of the proof reveals that this can be traced back to an
argument involving dominated convergence in the proof of \cite[Lemma
6.1]{h-b03}, thereby allowing Heath-Brown to employ \cite[Theorem 1]{h-b03}, which 
is not uniform in any of the relevant 
parameters.  Rather than using \cite[Theorem~1]{h-b03} to estimate the sums
$S(d,d')$ that occur in his analysis, however, it is possible to employ our
Theorem \ref{main1}. The advantage in doing so is that the
corresponding error term is completely uniform in
the parameters $d,d'$, thus circumventing the need for the
argument involving dominated convergence. 
Rather than labouring the details, we will content ourselves with merely recording the outcome of this
observation here.

\begin{cor}
One has $E(X)=X^2(\log X)^{-\eta/3+\ve}$ in \eqref{eq:bell}, for any $\ve>0$.
\end{cor}

In addition to the threefold $V\subset \bfP^5$ defined above, it turns out that the estimates in this
paper can play an important  role in analysing the arithmetic of other rational varieties.
Indeed, one of the motivating factors behind writing this paper has
been to prepare the way for a verification of the Manin conjecture for
certain surfaces of the shape
$$
x_1x_2=x_3^2, \quad x_3(ax_1+bx_2+cx_3) =x_3^2+x_4^2,
$$
in forthcoming joint work with Emmanuel Peyre.
These equations define singular del Pezzo surfaces of degree $4$ in
$\bfP^4$, of the type first considered by Iskovskikh. These are arguably the
most interesting examples of singular quartic del Pezzo surfaces since
they are the only ones for which weak approximation may fail.
On solving the first equation in integers, and substituting into the
second equation, one is led to consider the family of equations
$$
h^2y_1y_2(ay_1^2+by_2^2+cy_1y_2) =s^2+t^2,
$$
for $h$ running over a suitable range. Studying the distribution of
integer solutions to this system of equations therefore 
amounts to estimating sums of the shape
$$
\sum_{y_1,y_2}r(h^2y_1y_2(ay_1^2+by_2^2+cy_1y_2)),
$$
uniformly in $h$. By choosing $a,b,c$ such that $c^2-4ab$ is a square,
one can show that this sum is related to sums of the sort
\eqref{eq:Sj}, but for which Heath-Brown's original
normalisation conditions in \textsf{NH}$_0$ are no longer met. 
Thus we have found it desirable to generalise the work of \cite{h-b03} to the extent enjoyed in the
present paper.

As a final remark we note that at the expense of extra work further
generalisations of our main results are possible. For example it would
not be difficult to extend the work to deal with analogues of \eqref{eq:Sj} in which 
$r$ is replaced by a $r_\Delta$-function that counts
representations as norms of elements belonging to an arbitrary imaginary
quadratic field of discriminant $\Delta$. 
%In effect one has 
%$r_\Delta(n)=\sum_{d\mid n}\chi_\Delta(n)$, for a suitable Dirichlet
%character $\chi_\Delta$.

\begin{notat}
Throughout our work $\N$ will denote the set of positive
integers.  Moreover,  we will follow
common practice and allow the arbitrary small parameter $\varepsilon>0$ to
take different values at different parts of the argument.
All order constants will be allowed to depend on $\ve$. 
\end{notat}

\begin{ack} 
The authors are grateful to G\'erald Tenenbaum for 
discussions that have led to the overall improvement in the error term of
Theorem \ref{main0}, and to Emmanuel Peyre for discussions relating to
the interpretation of the constant in Theorem \ref{main3}. 
Part of this work was undertaken while the 
second author was visiting the first author at the {\em Universit\'e
  de Paris VII},  the hospitality and financial support 
of which is gratefully acknowledged.
\end{ack}

\section{Interpretation of the constant}\label{s:local}

Our task in this section is to establish Theorem \ref{main3}.  
We begin with some preliminary facts. 
Let $A\in\Z$ and let $\al\in \Z_{\geq 0}$. For any prime power $p^n$, we write 
\begin{equation}
  \label{eq:sum2}
  S_\al(A;p^n):=\#\{(x, y) \in(\Z/p^n\Z)^2: p^{\al}(x^2+y^2) \equiv A \bmod{p^n}\}.
\end{equation}
If $\al\leq n$ then it is not hard to see that
\begin{equation}
  \label{eq:al-0}
S_\al(A;p^n)=p^{2\al} S_0(A/p^\al;p^{n-\al}),
\end{equation}
when $\al\leq \nu_p(A)$ and $S_\al(A;p^n)=0$ otherwise. In the case
$\al=0$ we have
\begin{equation}
  \label{eq:s0-1}
S_0(A;p^n)=\left\{\begin{array}{ll}
p^n+np^n(1-1/p), &\mbox{if $\nu_p(A)\geq n$},\\
(1+\nu_p(A))p^{n}(1-1/p),
& \mbox{if $\nu_p(A)<n$},
\end{array}
\right.
\end{equation}
when $p\equiv 1\bmod{4}$. This formula has been employed by Heath-Brown
\cite[\S 8]{h-b03} in a similar context. When $p\equiv 3 \bmod{4}$, he
notes that
\begin{equation}
  \label{eq:s0-3}
S_0(A;p^n)=\left\{\begin{array}{ll}
p^{2[n/2]}, &\mbox{if $\nu_p(A)\geq n$},\\
p^{n}(1+1/p), & \mbox{if $\nu_p(A)<n$ and $2\mid \nu_p(A)$},\\
0, & \mbox{if $\nu_p(A)<n$ and $2\nmid \nu_p(A)$}.
\end{array}
\right.
\end{equation}
Finally, when $p=2$ and $n\geq 2$, we have
\begin{equation}
  \label{eq:s0-2}
S_0(A;2^n)=\left\{\begin{array}{ll}
2^{n+1}, & \mbox{if $2^{-\nu_2(A)}A\equiv 1\bmod{4}$,}\\
0, & \mbox{otherwise.}
\end{array}
\right.
\end{equation}
Note that Heath-Brown states this formula only for odd $A$ that are
congruent to $1$ modulo $4$, but the general case is easily checked.
Indeed, if $\nu=\nu_2(A)$, then one notes that
$2\mid \hcf(x,y)$ in the definition of $S_0(A;2^n)$ if $\nu\geq 2$,
and $2\nmid xy$ if $\nu=1$. In the
former case one therefore has $S_0(A;2^n)=4S_0(A/4;2^{n-2})$, and in
the latter case one finds that $S_0(A;2^n)=2^{n+1}$.

Let $L_1,\ldots,L_4\in\Z[x_1,x_2]$  be arbitrary linear forms, and
recall the definition \eqref{eq:rho*} of the determinant $\rho_*(\ma{h})$.
It follows from the multiplicativity of $\rho_*$ that 
$$
\frac{1}{\det \sfg_{\ma{D}}}
\prod_{p>2}\sigma_p =\prod_{p>2} c_p
$$
in the statement of Theorem \ref{main2}, with 
$$
c_p =\Big(1-\frac{\chi(p)}{p}\Big)^4
\sum_{n_i\geq 0}\frac{\chi(p)^{n_1+\cdots+n_4}}{\rho_*(p^{\max\{\nu_p(D_1),
    \nu_p(d_1)+n_1\}},\ldots, p^{\max\{\nu_p(D_4), \nu_p(d_4)+n_4\}})}.
$$
We claim that 
\begin{equation}
  \label{eq:claim-un}
  c_p=\omega_{\mal,\mmu}(p),
\end{equation}
for each $p>2$, where $\omega_{\mal,\mnu}(p)$ is given by 
\eqref{eq:def-sigp} and the values of $\mal,\mnu$ are as in the
statement of Theorem \ref{main3}. The proof of this claim will be in
two steps:  the case $p\equiv 1\bmod{4}$ and the
case $p\equiv 3\bmod 4$.

\begin{lem}\label{lem:p=1}
Let $p\equiv 1\bmod{4}$ be a prime. Then \eqref{eq:claim-un} holds.
\end{lem}

\begin{proof}
Let $A\in\Z$, and let $p\equiv 1\bmod{4}$ be a prime.
On combining \eqref{eq:s0-1} with \eqref{eq:al-0} it follows that
$$
S_\al(A;p^n)=(1+\nu_p(A)-\al)p^{n+\al}(1-1/p),
$$
provided that $\al\leq \nu_p(A)<n$. 
Our plan will be to fix $p$-adic valuations $\nu_i$ of
$L_i(\x)$,  and to then use this formula to count the resulting
number of $\ma{s},\ma{t} \in (\Z/p^n\Z)^{4}$ in $N_{\mal,\mmu}(p^n)$. 
Note that we must have 
$$
\nu_i\geq M_i:=\max\{\lambda_i, \mu_i\}.
$$ 
It follows that
\begin{align*}
N_{\mal,\mmu}(p^n)
=&p^{4n+\la_1+\cdots+\la_4}
\Big(1-\frac{1}{p}\Big)^4\sum_{\nu_i\geq M_i}M_{\mnu}(p^n)\prod_{1\leq
  i\leq 4}(1+\nu_i-\lambda_i)\\
&\quad+O(n^4p^{5n}),
\end{align*}
where $M_{\mnu}(p^n)$ counts the number of $\x \bmod{p^n}$ such that 
$p^{\mu_i}\mid L_i(\x)$ and $\nu_p(L_i(\x))=\nu_i$.
But then 
\begin{align*}
M_{\mnu}(p^n)
&=\sum_{\ma{e}\in \{0,1\}^4} (-1)^{e_1+\cdots +e_4} \#\big\{\x
\bmod{p^n}: p^{\max\{\nu_i+e_i, \mu_i\}}\mid L_i(\x)\big\} \\
&=\sum_{\ma{e}\in \{0,1\}^4} (-1)^{e_1+\cdots +e_4} \#\big\{\x
\bmod{p^n}: p^{\nu_i+e_i}\mid L_i(\x)\big\} \\
&=p^{2n}\sum_{\ma{e}\in \{0,1\}^4} \frac{(-1)^{e_1+\cdots +e_4}}{
\rho_*(p^{\nu_1+e_1}, \ldots, p^{\nu_4+e_4})}.
\end{align*}
Making the change of variables $n_i=\nu_i+e_i-\lambda_i$, and noting that
$\nu_i+e_i\geq M_i+e_i\geq M_i$, we therefore deduce that
\begin{align*}
\sigma_{\mal,\mmu}(p)
=&\Big(1-\frac{1}{p}\Big)^4\sum_{n_i\geq M_i-\lambda_i }
\rho_*(p^{\la_1+n_1}, \ldots, p^{\la_4+n_4})^{-1}\\
&\quad \times
\sum_{0\leq e_i\leq \min\{1,\la_i+n_i-M_i\}} (-1)^{e_1+\cdots +e_4}
\prod_{1\leq i\leq 4}(1+n_i-e_i).
\end{align*}
Now it is clear that 
$$
\sum_{0\leq e\leq \min\{1,\la+n-M\}} 
\hspace{-0.2cm}
(-1)^{e}
(1+n-e)=
\left\{
\begin{array}{ll}
1,& \mbox{if $\la+n-M\geq 1$},\\
1+M-\la,& \mbox{if $\la+n-M=0$}.
\end{array}
\right.
$$
Since $1+M-\la=\#\Z\cap [0,M-\lambda]$, a little thought reveals that
\begin{align*}
\sigma_{\mal,\mmu}(p)
&=\Big(1-\frac{1}{p}\Big)^4\sum_{n_i\geq 0}
\rho_*(p^{\max\{M_1,\la_1+n_1\}}, \ldots,
p^{\max\{M_4,\la_4+n_4\}})^{-1}\\
&=\Big(1-\frac{1}{p}\Big)^4\sum_{n_i\geq 0}
\rho_*(p^{\max\{\mu_1,\la_1+n_1\}}, \ldots, p^{\max\{\mu_4,\la_4+n_4\}})^{-1}.
\end{align*}
This completes the proof of the lemma.
\end{proof}

\begin{lem}
Let $p\equiv 3\bmod{4}$ be a prime. Then \eqref{eq:claim-un} holds.
\end{lem}

\begin{proof}
Let $\al\in\Z_{\geq 0}$ and $A\in\Z$,  and recall the definition \eqref{eq:sum2} of
$S_\al(A;p^n)$. Combining \eqref{eq:s0-3} with \eqref{eq:al-0}, and arguing precisely as
in the proof of Lemma~\ref{lem:p=1}, we conclude that
\begin{align*}
N_{\mal,\mmu}(p^n)
=&p^{6n+\la_1+\cdots+\la_4}\Big(1+\frac{1}{p}\Big)^4\sum_{\colt{\nu_i\geq M_i}{2\mid
    \nu_i-\la_i}}
\sum_{\ma{e}\in \{0,1\}^4} \frac{(-1)^{e_1+\cdots +e_4}}{
\rho_*(p^{\nu_1+e_1}, \ldots, p^{\nu_4+e_4})}\\
&\quad +O(n^4p^{5n}).
\end{align*}
Making the change of variables $n_i=\nu_i+e_i-\lambda_i$, it follows that
\begin{align*}
\sigma_{\mal,\mmu}(p)
=&\Big(1+\frac{1}{p}\Big)^4\sum_{n_i\geq M_i-\lambda_i }
\rho_*(p^{\la_1+n_1}, \ldots, p^{\la_4+n_4})^{-1}\\
&\quad \times
\sum_{\colt{0\leq e_i\leq \min\{1,\la_i+n_i-M_i\}}{e_i\equiv
    n_i \bmod{2}}} (-1)^{e_1+\cdots +e_4}.
\end{align*}
This time we find that the summand can be expressed in terms of 
$$
\sum_{\colt{0\leq e\leq \min\{1,\la+n-M\}}{e\equiv
    n \bmod{2}}} 
\hspace{-0.2cm}
(-1)^{e}
=
\left\{
\begin{array}{ll}
(-1)^{n},& \mbox{if $\la+n-M\geq 1$},\\
1,& \mbox{if $\la+n-M=0$ and $2\mid M-\lambda$},\\
0,& \mbox{if $\la+n-M=0$ and $2\nmid M-\lambda$}.
\end{array}
\right.
$$
Since $\sum_{0\leq n\leq M-\lambda}(-1)^{n}$ is equal to $1$ if
$M-\lambda$ is even, and $0$ otherwise, we conclude that
\begin{align*}
\sigma_{\mal,\mmu}(p)
&=\Big(1+\frac{1}{p}\Big)^4\sum_{n_i\geq 0}
\frac{(-1)^{n_1+\cdots+n_4}}{\rho_*(p^{\max\{\mu_1,\la_1+n_1\}}, \ldots,
p^{\max\{\mu_4,\la_4+n_4\}})}.
\end{align*}
This completes the proof of the lemma.
\end{proof}

We now turn to the $2$-adic density, for which we claim that
\begin{equation}
  \label{eq:claim-deux}
\delta_{j,k}(\A)=\omega_{j,k,\bd}(2),
\end{equation}
where $\delta_{j,k}(\ma{A})$ is given by \eqref{eq:sig2} and 
$\omega_{j,k,\bd}(2)$ is given  by \eqref{eq:eq:def-dig2}.
On recalling the definition \eqref{eq:Ar} of $E_n$, 
it follows from \eqref{eq:s0-2} that
\begin{align*}
N_{j,k,\bd}(2^n)
=&2^{4n+4}
\#\left\{
\x\in \Z/2^n\Z : \begin{array}{l}
L_i(\x)\in d_iE_n\\
x_1 \equiv 1 \bmod{4}, ~x_{2}\equiv j \bmod{2}
\end{array}
\right\}.
\end{align*}
But then 
\begin{align*}
\omega_{j,k,\bd}(2)
&=
\lim_{n\to\infty} \frac{1}{2^{2n-4}} 
\#\left\{
\x\in \Z/2^n\Z : \begin{array}{l}
L_i(\x)\in d_iE_n\\
x_1 \equiv 1 \bmod{4}, ~x_{2}\equiv j \bmod{2}
\end{array}
\right\},
\end{align*}
which is just $\delta_{j,k}(\ma{A})$. This completes the proof of \eqref{eq:claim-deux}.

Finally we turn to the archimedean density $\omega_{\mcal{R}}(\infty)$
of points on the variety \eqref{eq:torsor} for which $\x\in\mcal{R}$.
We claim that
\begin{equation}
  \label{eq:claim-trois}
\omega_{\mcal{R}}(\infty)
=\pi^4 \meas(\mcal{R}).
\end{equation}
Our assumptions on $L_1,\ldots,L_4,\mcal{R}$ imply that
$L_i(\x)>0$ for all $\x\in\mcal{R}$.
To begin with, it is clear that
$$
\omega_{\mcal{R}}(\infty)=2^{8}\omega_{\mcal{R}}^+(\infty),
$$ 
where $\omega_{\mcal{R}}^+(\infty)$ is defined as for
$\omega_{\mcal{R}}(\infty)$, but with the additional constraint that $s_i,t_i>0$.
We will calculate $\omega_{\mcal{R}}^+(\infty)$
by parametrising the points via the $t_i$, using the Leray form.
In this setting the Leray form is given by 
$$
(2^4t_1t_2t_3t_4)^{-1}\d s_1\cdots \d s_4\d x_1\d x_2.
$$ 
On making the
substitution $t_i=\sqrt{d_i^{-1} L_i(\x)-s_i^2}$, and noting that
$$
\int_0^{\sqrt{A}} \frac{\d s}{\sqrt{A-s^2}}=\frac{\pi}{2},
$$
we therefore conclude that
\begin{align*}
\omega_{\mcal{R}}(\infty)
&=2^{4}
\int_{\x\in \mcal{R}} \Big(\prod_{1\leq i\leq
  4}\int_0^{\sqrt{d_i^{-1} L_i(\x)}}\frac{\d s}{\sqrt{d_i^{-1} L_i(\x)-s^2}}\Big)\d x_1\d x_2\\
&=\pi^4 \meas(\mcal{R}),
\end{align*}
as required for \eqref{eq:claim-trois}.

Bringing together \eqref{eq:claim-un}, \eqref{eq:claim-deux} and
\eqref{eq:claim-trois}, we easily deduce the statement of Theorem \ref{main3}.

\section{The $2$-adic densities}\label{2-adic} 

In this section we explicitly calculate the value of the $2$-adic
densities $\del_{j,k}(\A)=\del_{j,k}(\A,\ma{d})$
in \eqref{eq:sig2}. In effect this will simplify the process of
deducing Theorem \ref{main2}. Let $L_1,\ldots,L_4\in \Z[x_1,x_2]$ be
arbitrary linear forms that satisfy any of the normalisation
conditions from the introduction, with $L_3, L_4$ given by
\eqref{L3L4}. In particular, it is clear that 
there exist integers $k_3,k_4\geq 0$ such that
\begin{equation}
   \label{eq:L34}
2^{-k_3}L_3(\x)=2^{\mu_3}a_3'x_1+2^{\nu_3}b_3'x_2, \quad 
2^{-k_4}L_4(\x)=2^{\mu_4}a_4'x_1+2^{\nu_4}b_4'x_2, 
\end{equation}
for integers $a_i',b_i'$ such that
\begin{equation}
   \label{eq:aibi} 
a_3'a_4'b_3'b_4'(a_3'b_4'-a_4'b_3')\neq 0, \quad 2\nmid a_3'a_4'b_3'b_4',
\end{equation}
and integers $\mu_i,\nu_i\geq 0$ such that
\begin{equation}
   \label{eq:munu}
\mu_3\nu_3=\mu_4\nu_4=0.
\end{equation}
We are now ready to proceed with the calculation
of $\delta_{j,k}(\A)$, whose value will depend intimately on $j,k$, $\ma{d}$
 and the values of the coefficients in \eqref{eq:L34}.
The calculations in this section are routine and so we will be
brief. In fact we will meet these calculations again in \S \ref{s:t2}
under a slightly different guise.

Recall the definition \eqref{eq:Ar} of $E_n$ for any $n \in \N$, and
the definition \eqref{eq:sig2} of $\delta_{j,k}(\A)$,
for $L_1,\ldots,L_4,\mcal{R}$ satisfying \textsf{NH}$_k(\ma{d})$. 
When $k=0$, it easily follows from our normalisation conditions that
$L_i(\x) \in d_i E_n$  
for any integer vector $\x$ such that $x_1 \equiv 1 \bmod{4}$. Hence
\begin{equation}
   \label{eq:C}
\begin{split}
\delta_{j,0}(\A) &=
\delta_j,
\end{split}
\end{equation}
in the notation of \eqref{eq:dj}. 

Let us now suppose that $j=k=1$. Then clearly 
\begin{equation}\label{eq:sig2'}
\delta_{1,1}(\A)=\lim_{n\to\infty} \frac{1}{2^{2n-4}} \#\Big\{
\x\in(\Z/2^n\Z)^2 :
\begin{array}{l}
x_1\equiv 1 \bmod 4, ~2\nmid x_2 \\  
d_3L_3(\x), d_4L_4(\x) \in E_n
\end{array}\Big\}.
\end{equation}
It follows from \eqref{eq:munu} that at most two of
$\mu_3,\mu_4,\nu_3,\nu_4$ can be non-zero.  An easy calculation shows that
\begin{equation}
   \label{delta1}
  \delta_{1,1}(\A)=\left\{
\begin{array}{ll}
1, & \mbox{if $b_3'd_3-2^{\mu_3}\equiv b_4'd_4-2^{\mu_4} \bmod 4 $,}\\ 
0 , & \mbox{otherwise},
\end{array}
\right.
\end{equation}
when $\nu_3=\nu_4=0$ and $\mu_3,\mu_4\geq 1$.
Similarly, we deduce that
$$
  \delta_{1,1}(\A)=\left\{
\begin{array}{ll}
2, & \mbox{if $ a_j'  \equiv d_j-2^{\nu_j} \bmod 4$ for $j=3,4$},\\
0 , & \mbox{otherwise},
\end{array}
\right.
$$
when $\mu_3=\mu_4=0$ and $\nu_3,\nu_4\geq 1$.
Let $j_1,j_2$ denote distinct elements from the set $\{3,4\}$. Then it
follows from \eqref{eq:sig2'} that
\begin{equation}
   \label{delta3}
  \delta_{1,1}(\A)=\left\{
\begin{array}{ll}
1, & \mbox{if $ a_{j_1}'  \equiv d_{j_1}-2^{\nu_{j_1}} \bmod 4 $},\\
0 , & \mbox{otherwise},
\end{array}
\right.
\end{equation}
when $\mu_{j_1}=\nu_{j_2}=0$ and $\mu_{j_2}, \nu_{j_1}\geq 1$.
Still with the notation $\{j_1,j_2\}=\{3,4\}$, a simple calculation
reveals that
\begin{equation}
   \label{delta4}
\delta_{1,1}(\A)=\left\{
\begin{array}{ll}
1, & \mbox{if $a_{j_2}'\equiv d_{j_2}-2^{\nu_{j_2}}\ \bmod 4$},\\  
0, & \mbox{otherwise},
\end{array}
\right.
\end{equation}
when $\mu_3=\mu_4=\nu_{j_1}=0$ and $\nu_{j_2}\geq 1$.
In performing this calculation it is necessary to calculate the
contribution to the right hand side of \eqref{eq:sig2'} for fixed
values of $n$ and fixed $2$-adic valuation $\xi$ of 
$a_3'x_1+b_3'x_2$, before then summing over all possible values of $\xi\geq
1$. In a similar fashion, one finds 
\begin{equation}
   \label{delta5}
\delta_{1,1}(\A)=1/2,
\end{equation}
when $\nu_3=\nu_4=\mu_{j_1}=0$ and $\mu_{j_2}\geq 1$.
It remains to handle the case in which all the $\mu_j,\nu_j$ are
zero. For this we set 
\begin{equation}
  \label{eq:v}
v:=\nu_2(a'_3b'_4-a'_4b'_3),
\end{equation}
which must be a positive integer, since $a_j',b_j'$ are all odd.
Thus we have 
\begin{equation}
   \label{delta6}
\delta_{1,1}(\A)=\left\{
\begin{array}{ll}
1/2, & \mbox{if  $v=1$},\\
1-3/2^{v}, & \mbox{if  $v\geq 2$ and $b_3'd_3 \equiv b_4'd_4 \bmod 4$},\\ 
3/2^{v}, & \mbox{if  $v\geq 2$ and $b_3'd_3 \equiv -b_4'd_4 \bmod 4$}, 
\end{array}
\right.
\end{equation}
when $\mu_3=\mu_4=\nu_3=\nu_4=0$.

When $j\neq 1$, and $k\neq 0$, we will find it convenient to
phrase our formulae for $\del_{j,k}(\A)$ in terms of $\del_{1,k}(\A)$. 
We claim that
\begin{equation}
  \label{eq:1811.1}
\delta_{0,k}(\A)=\sum_{\xi=1}^\infty \frac{ \delta_{1,k}(\A\M_\xi) }{2^{\xi}},\quad
\delta_{*,k}(\A)=\sum_{\xi=0}^\infty \frac{ \delta_{1,k}(\A\M_\xi)
}{2^{\xi}}
\end{equation}
when $k=1$ or $2$, where 
\begin{equation}\label{defMxi}
\M_\xi:=\Big(
\begin{array}{cc}
1&0\\
0 &2^\xi
\end{array}
\Big).
\end{equation}
Here the formula for $\del_{0,k}(\A)$ is not hard to establish, and follows on
extracting the $2$-adic valuation of $x_2$ in \eqref{eq:sig2}. The
formula for $\del_{*,k}(\A)$ follows on noting that $\del_{*,k}(\A)=
\del_{0,k}(\A)+\del_{1,k}(\A)$.  Finally, we express $\del_{1,2}(\A)$
in terms of $\del_{*,1}(\A)$ via the transformation
\begin{equation}  
  \label{defM}
\M_{c,d_2}:=\Big(
\begin{array}{cc}
1&0\\
\kappa+4c &4
\end{array}
\Big),
\end{equation}
where $\kappa=\pm 1$ denotes the residue modulo $4$ of $d_2$, and  
$c\in \{0,1,2\}$ is any parameter we care to choose.
It is not hard to see that 
\begin{equation}
  \label{eq:1811.2}
\delta_{1,2}(\A)=\frac{\del_{*,1}(\A\M_{c,d_2})}{4},
\end{equation}
using the fact that $x_1\equiv  1 \bmod{4}$ and $ x_2 \equiv d_2 \bmod{4}$.

\section{Proof of Theorem \ref{main0}}

Our proof follows that given by Heath-Brown for \cite[Theorem
 1]{h-b03}, but with extra care taken to keep track of the error 
 term's dependence on $L_1, \ldots, L_4$ and $\RR$.
Our improvement in the exponent of $\log X$ will emerge through a 
modification of the the final stages of the argument.

Let $X\RR_4:=\{\x\in\Z^2\cap X\RR: x_1\equiv 1 \bmod 4\}$, and for
given $\ma{d}\in \N^4$ let $\RR(\ma{d})\subseteq \RR$ denote a convex
region depending on $\ma{d}$. We write $X\RR_4(\ma{d})$ for the set 
$ \{\x\in\Z^2\cap X\RR(\ma{d}): x_1\equiv 1 \bmod 4\}$. The first step of the argument involves
modifying the ``level of distribution'' result that is employed by Heath-Brown \cite[Lemma 2.1]{h-b03}.

\begin{lem}\label{lem21} 
Let $X\geq1$ and  $Q_1,Q_2,Q_3,Q_4\ge 2$. Write
$Q=\max_i Q_{i}$ and $V=Q_1 Q_2 Q_3 Q_4.$
 Then there is an absolute constant $A>0$ such that
\begin{align*}
 \sum_{\tolt{\ma{d}\in\N^4}{d_i\le Q_i}{2\nmid d_i}} 
\left|\#\big(\sfg_{\ma{d}}\cap X\RR_4(\ma{d}) \big)
 -\frac{\meas(\RR(\ma{d}))X^2}{4\det\sfg_{\ma{d}}}\right|\\
  \ll  L_\infty^{\varepsilon}r_\infty X( V^{1/2}(\log Q)^{A}+ Q )+V.
\end{align*}
\end{lem}
 
\begin{proof}
We appeal to work of Daniel
\cite[Lemma 3.2]{daniel}. This gives
\begin{equation}\label{eqdaniel}
\left|\#\big(\sfg_{\ma{d}}\cap X\RR_4(\ma{d}) \big)
 -\frac{\meas(\RR(\ma{d}))X^2}{4\det\sfg_{\ma{d}}}\right|\
  \ll r_\infty\frac{X}{|\ma{v}|}+1,
\end{equation}
for some vector $\ma{v} \in \sfg_{\ma{d}}$ with coprime coordinates,
such that
$$
|\ma{v}|\ll (\det \sfg_{\ma{d}})^{1/2}\leq (d_1d_2d_3d_4)^{1/2}\leq
V^{1/2}.
$$ 
The contribution from the second term in \eqref{eqdaniel} is clearly
$O(V)$. To complete the proof of the lemma it will suffice to show
that
\begin{equation}
  \label{eq:2211.1}
\sum_{\colt{\ma{d}\in\N^4}{d_i\le Q_i}}\frac{1}{|\ma{v}|}\ll L_\infty^{\varepsilon}   (
V^{1/2}(\log Q)^{A}+ Q ),
\end{equation}
for some absolute constant $A>0$.

Let  $\sigma_1$ denote the contribution from the case in which
$L_1(\ma{v})\cdots L_4(\ma{v})\neq 0$, and let $\sigma_2$ denote the
remaining contribution.  We then have
$$
\sigma_1\leq \sum_{\colt{|\ma{v}|\ll
V^{1/2}}{L_i(\v)\neq 0}}\frac{1}{|\ma{v}|}\sum_{\tolt{\ma{d}\in\N^4}{d_i\le Q_i}{d_i\mid L_i(
\ma{v})}}1 \ll
L_\infty^{\varepsilon}\tau(F(\ma{v})),
$$ 
where $\tau$ is the divisor function and $F$ is a primitive
binary form that is proportional to $L_1\cdots L_4.$ 
A simple application of \cite[Corollary 1]{nair} now reveals that
there exists a constant $A>0$ such that
$$ 
\sum_{|\ma{v}|\leq x }\tau(F(\ma{v}))\ll L_\infty^\varepsilon x^2(\log
x)^A.
$$
We therefore obtain the estimate $
\sigma_1\ll L_\infty^{\varepsilon} V^{1/2}(\log Q)^{A},
$
on carrying out a dyadic summation for the range of $\ma{v}$, which is
satisfactory for \eqref{eq:2211.1}.

Turning to a bound for $\sigma_2$, we suppose that $i_0\in
\{1,2,3,4\}$ is an index for which
$L_{i_0}(\ma{v})=a_{i_0}v_1+b_{i_0}v_2=0$. Since
$\hcf(v_1,v_2)=1$, we have $v_1\mid b_{i_0}$ and $v_2\mid a_{i_0}$. If
$j\neq {i_0}$, then $L_j(\ma{v})\neq 0$ because $L_{i_0}$ and  $L_j$
aren't proportional. Moreover, we have $|L_j(\ma{v})|\leq 2
L_\infty^2$ and the number of possible values of $L_j(\ma{v})$ is
bounded by $O( L_\infty^\varepsilon)$. Since  $d_j\mid L_j(\ma{v})$,
the number of available $d_j$ is $O( L_\infty^{\varepsilon})$, whereas
the number of $d_{i_0}$ is bounded by $Q_{i_0}\leq Q$. Thus it follows
that 
$
\sigma_2\ll L_\infty^{\varepsilon} Q,
$ 
which therefore completes the proof of \eqref{eq:2211.1}.
\end{proof}

Recall the definition \eqref{eq:r'} of $r'=r'(L_1,\ldots,L_4,\RR)$. It
will be convenient to set
$$
X':=r'X
$$
in what follows, and to assume that 
$r'X^{1-\ve}\geq 1$. In particular this ensures that $\log X' \gg \log
X$.

Our next task is to establish a uniform version of
\cite[Lemma 3.1]{h-b03}. The reader is recommended to consult \cite{h-b03} for full
details of the ensuing argument, since we will only stress those parts 
where modification is needed.
When $0<m\leq X'$ and $m\equiv 1\bmod{4}$, we may write
\begin{align*}
r(m)=4\sum_{\colt{d\mid m}{d\leq {X'}^{1/2}}}\chi (d)+
4\sum_{\colt{e\mid m}{m>e {X'}^{1/2}}}\chi (e)=4A_+(m)+4A_-(m),
\end{align*}
say,
as in \cite{h-b03}.
This will be employed with $m=L_i(\ma{x})$ for $1\leq i\leq 3$. The conditions
$L_i(\ma{x})\equiv v_1 \bmod{4}$ and $v_1\equiv 1 \bmod{4}$ yield 
$m\equiv 1 \bmod{4}$. In a similar fashion, we may write 
$$
r(m)  =4B_+(m)+4C(m)+4B_-(m),
$$
under the same hypotheses on $m$, with 
$$
B_+(m):=\sum_{\colt{d\mid m}{d\leq Y}}\chi (d),\quad C(m):=
\sum_{\colt{d\mid
m}{Y<d\leq
 X'/Y}}\!\!\!\!\!\chi (d),\quad
B_-(m):=\sum_{\colt{e\mid m}{m>e  X'/Y}}\!\!\!\!\chi (e).
$$
Here $1\leq Y\leq {X'}^{1/2}$ is a parameter to be chosen in due 
course.  This 
formula will be used with $m=L_4(\ma{x})$. The variable $e$ in $A_-(L_i(\x))$ and $B_-(L_4(\x))$
will satisfy $e\leq {X'}^{1/2}$ and $e\leq Y$, respectively. 

On writing 
$$
S_{\pm,\pm,\pm,\pm}
:=\sum_{\ma{x}\in
X\RR_4}A_\pm(L_1(\ma{x}))A_\pm(L_2(\ma{x}))A_\pm(L_3(\ma{x}))
B_\pm(L_4(\ma{x})),
$$
we obtain
$$S_*(X)=4S_0+4^4\sum S_{\pm,\pm,\pm,\pm},$$
which is the analogue of \cite[Eq. (3.4)]{h-b03}.
Let us consider the sum $S_{+,+,-,-}$, the other $15$ sums being handled 
similarly. Write $Q_1=Q_2=Q_3={X'}^{1/2}$ and $Q_4=Y$. Then 
$$
S_{+,+,-,-}= \sum_{\colt{\ma{d}\in\N^4}{d_i\le Q_i}}\chi
(d_1d_2d_3d_4)\#\bigl(\sfg_{\ma{d}}\cap X\RR_4(\ma{d})
\bigr),
$$
where $\RR(\ma{d}):=\{ \x\in \RR: L_3(\x)>d_3{X'}^{1/2},
~L_4(\x)>d_4X'/Y\}$.
An application of Lemma \ref{lem21} therefore implies that
\begin{equation}
  \label{eq:2311.1}
S_{+,+,-,-}= \sum_{\colt{\ma{d}\in\N^4}{d_i\le Q_i}}\chi
(d_1d_2d_3d_4)  
\frac{ \meas(\RR(\ma{d}))X^2}{4\det\sfg_{\ma{d}}} +O(T),
\end{equation}
with 
$$
T:=L_\infty^\varepsilon r_\infty  X 
{X'}^{3/4}Y^{1/2}(\log X')^{A} + {X'}^{3/2}Y,
$$
and $A\geq 2$. 
Choosing $Y={X'}^{1/2}/(\log X')^{2A+2}$, we obtain 
$$
T\ll  \frac{L_\infty^\varepsilon r_\infty r'X^2}{\log X'}+ 
\frac{{r'}^2X^2}{(\log X')^{2A+2}}.
$$
We claim that it is possible to take
\begin{equation}
  \label{eq:2311.2}
T\ll  \frac{L_\infty^\varepsilon r_\infty r'X^2}{\log X}
\end{equation}
in \eqref{eq:2311.1}. When $r'\leq r_\infty(\log X')^{2A+1}$ this is
trivial, since the assumption $r'X^{1-\ve}\geq 1$ yields
$\log X' \gg \log X$.
Suppose now that $r'>r_\infty(\log X')^{2A+1}\gg r_\infty (\log 
X)^{2A+1}$. 
Then on returning to the original
definition of $S_{\pm,\pm,\pm,\pm}$, it follows from
an easy application of \cite[Corollary 1]{nair} that
\begin{align*}
S_{+,+,-,-}
\ll \sum_{\ma{x}\in
X\RR_4}\tau\big(L_1(\ma{x})L_2(\ma{x})L_3(\ma{x})L_4(\ma{x})\big)
&\ll L_\infty^\ve r_\infty^2 X^2 (\log X)^4\\
&\ll L_\infty^\ve r_\infty r' X^2 (\log X)^{3-2A}. 
\end{align*}
Thus we may certainly take \eqref{eq:2311.2} in 
\eqref{eq:2311.1} in this case too. 
 
Although we will omit the details here, it is easy to modify the
argument of  \cite{h-b03} to deduce that the main term in
\eqref{eq:2311.1} is
$$
\frac{\pi^4 \meas(\mcal{R})X^2}{4^5} \prod_{p>2}\sigma_p^*
+O\big(L_\infty^\ve r_\infty r' X^{79/40+\ve}\big),
$$
and similarly for all the $S_{\pm,\pm,\pm,\pm}$.
Bringing all of this together we have therefore established the
following result.

\begin{lem}\label{lem31}
Assume that $r'X^{1-\ve}\geq 1$. Then we have 
$$
S_*(X)=4\pi^4 \meas(\mcal{R})X^2 \prod_{p>2}\sigma_p^*+4S_0+
O\Big(\frac{L_\infty^\varepsilon r_\infty r'X^2}{\log X}\Big),
$$
where
$$
S_0:=\sum_{\ma{x}\in
X\RR_4}r(L_1(\ma{x}))r(L_2(\ma{x}))r(L_3(\ma{x}))
C(L_4(\ma{x})).
$$
\end{lem}

To conclude our treatment of $S_*(X)$ we must estimate 
$S_0$.  Let
$$
\mcal{B}:=\{ m\in\Z: \exists d\mid m, Y<d\leq X'/Y\}\cap \{m\in \Z:\exists
\ma{x}\in X\RR_4, L_4(\ma{x})=m\}.
$$
Then as in \cite{h-b03}, we write
\begin{equation}
  \label{eq:2311.3}
S_0\ll \sum_{m\in  \mcal{B}}  S_0(m)|C(m)|,
\end{equation}
where
$$
S_0(m):=\sum_{\ma{x}\in
  \mcal{A}(m)}r(L_1(\ma{x}))r(L_2(\ma{x}))r(L_3(\ma{x}))
$$
and $\mcal{A}(m):=\{ \ma{x}\in X\RR_4: L_4(\ma{x})=m\}. $ We proceed
to establish the following estimate

\begin{lem}\label{l:2311.4}
There exists an absolute constant $c_0>0$ such that
$$
S_0(m)\ll L_\infty^\varepsilon  r_\infty X(\log\log X')^{c_0}.
$$
\end{lem}

\begin{proof}
We begin by recalling the notation used in \cite{h-b03}, with only
very minor modifications. Suppose that  $L_i(\x)=a_ix_1+b_ix_2$ with
$a_i\equiv 1 \bmod{4}$ and $b_i\equiv 0 \bmod{4}$. Then we have 
$x_1=(m-b_4x_2)/a_4$ and
$$
L_i(\x)=\frac{A_im+B_in}{a_4}=L_i'(m,n),
$$ 
with $A_i=a_i,$ $n=x_2$ and $B_i=a_4b_i-a_ib_4$. Its crucial to
observe that $B_1B_2B_3\neq 0$ since none of  $L_1,L_2,L_3$ are
proportional to $L_4$. We will use the inequality 
$r(L_i'(m,n))\leq r(a_4(A_im+B_in))$. Note that
$$
a_4(A_im+B_in)=a_4\hcf(A_im,B_i)(A_i'(m)+B_i'n)
$$ 
with $B_i':=B_i /\hcf(A_im,B_i)$ and $A'_i(m)=A_im/\hcf(A_im,B_i)$. In
particular these coefficients are coprime.
Write
$$
H= a_4^3B_1B_2B_3 \prod_{1\leq i\neq j\leq 3}|a_ib_j-a_jb_i|,
$$
and introduce the multiplicative function $r_1$, given by
$$
r_1(p^\nu)=\left\{
\begin{array}{ll}
\nu+1, & \mbox{if $p\mid H$},\\
r(p^\nu) , & \mbox{otherwise}.
\end{array}
\right.
$$
Then we have
\begin{align*}
r(L_1(\ma{x}))r(L_2(\ma{x}))r(L_3(\ma{x}))
&\leq r(a_4^3)r(B_1B_2B_3)\prod_{i=1}^3r_1(A_i'(m)+B_i'n)
\\&\ll L_\infty^\varepsilon r_1\big(G_m(n)\big),
\end{align*}
where $G_m(X):=\prod_{i=1}^3 (A_i'(m)+B_i'X)$ is a primitive cubic
polynomial with coefficients bounded in size by $O(L_\infty^6)$.

Bringing all of this together we have so far shown that 
$$
S_0(m)\ll L_\infty^\varepsilon\sum_{n\leq r_\infty X}r_1(G_m(n)).
$$
It now follows from \cite[Theorem 2]{nair} that there exists an
absolute constant $c_0>0$ such that
$$
S_0(m)
\ll L_\infty^\varepsilon r_\infty X(\log\log m)^{c_0}
\ll
L_\infty^\varepsilon r_\infty X(\log\log X')^{c_0},
$$
since visibly $S_0(m)=0$ unless $m \leq r'X=X'$. This completes the
proof of the lemma.
\end{proof}

It remains to consider the sum $\sum_{m\in \mcal{B}}|C(m)|$ in
\eqref{eq:2311.3}. It is precisely at this point that our
argument diverges from the proof of Heath-Brown.  
Define the function 
\begin{equation}
  \label{eq:Q}
Q(\lambda):=\lambda\log\lambda-\lambda+1.
\end{equation}
Then we have
$$
\max_{\lambda\in (1,2)}\min\{ Q(\lambda) ,2Q(\lambda
/2)\} =Q(1/\log 2)=2Q(1/(2\log 2))=\eta, 
$$
where $\eta$ is given by \eqref{defeta}.
With this in mind, we have the following result.

\begin{lem}\label{l:2311.5}
We have 
$$
\sum_{m\in \mcal{B}}|C(m)|\ll \frac{r'X(\log\log X')^{9/4}}{(\log X')^{\eta}}.
$$
\end{lem}

In view of the fact that $|C(m)| \geq 1$ for any $m$ such that
$C(m)\neq 0$, we deduce from \cite[part (ii) of Theorem 21]{HT} that 
one cannot hope to do much better than this estimate, since
up to multiplication by powers of $\log\log X'$ 
it is the true order of magnitude of the set $\mcal{B}$.

\begin{proof}[Proof of Lemma \ref{l:2311.5}]
Define the sum 
$$
\sigma(X';v):=\sum_{1\le m\le X'}|C(m)|^2v^{\Omega(m)},
$$
for any real number $v\in [0,1]$, where $\Omega(m)$ denotes the
total number of prime factors of $m$.
A crucial ingredient in the proof of Lemma \ref{l:2311.5} will be the estimate
\begin{equation}\label{majsigma}
\sigma(X';v) \ll X'(\log\log X')^3 (\log Y)^{2v-2}.
\end{equation} 
This coincides with the estimate obtained by Heath-Brown in \cite[\S
5]{h-b03} when $v=1$.  To establish \eqref{majsigma} we begin by expanding
$|C(m)|^2$ and drawing out the highest common factor of the variables
involved. This gives
\begin{align*}
|C(m)|^2=\sum_{h \mid m} \chi(h^2) 
\sum_{\colt{k_1\mid m/h}{Y<hk_1\leq
 X'/Y}}\chi(k_1)
\sum_{\tolt{k_2\mid m/hk_1}{Y<hk_2\leq
 X'/Y}{\hcf(k_1,k_2)=1}}\chi(k_2).
\end{align*}
Once substituted into $\sigma(X';v)$, let us write $\sigma_1$
for the overall contribution from $h \leq Y$ and $\sigma_2$ for
the contribution from the remaining $h$. Note that
we must have $Y<h \leq X'/Y$ in $\sigma_2$, since $h \leq hk_1\leq X'/Y$.  Write $Z:=X'/Y$. Then we have 
\begin{align*} 
\sigma_1
&=
\sum_{h\leq Y}\chi(h^2)v^{\Omega( h)}\sum_{Y/h< k_1 \leq Z/h}
\chi(k_1) v^{\Omega(k_1)}\sum_{n< Z/k_1} v^{\Omega(n)}
\sum_{k_2}\chi(k_2)v^{\Omega(k_2 )},
\end{align*}
where the final summation is over integers $k_2$ such that $\hcf(k_1,k_2)=1$
and $Y/h<k_2\leq \min \{Z/h, X'/hk_1n\}$. 
Here the inequality $n< Z/k_1$ follows from the two inequalities
$n\leq X'/hk_1k_2$ and $hk_2>Y$.  We will need the basic estimates 
\begin{equation} 
  \label{eq:112.2}
\sum_{n\leq x} v^{\Omega(n)}\ll x(\log 2x)^{v-1}, 
\end{equation}
and 
\begin{equation}
  \label{eq:112.1}
  \sum_{\colt{k_2\leq x}{\hcf(k_1,k_2)=1}}\chi(k_2)v^{\Omega(k_2  )}
\ll \tau(k_1)x\exp\{-3\sqrt{\log 2x}\},
\end{equation}
for any $v\in [0,1]$. 
When $k_1=1$ the latter bound follows from the fact that the 
corresponding 
Dirichlet series can be embedded holomorphically into a zero-free region
for $L(s,\chi)$. The general case then follows from an application of 
M\"obius inversion.

For fixed values of $h$ and $k_1$, 
\eqref{eq:112.1} and \eqref{eq:112.2} imply that the 
overall contribution to $\sigma_1$ from $n \leq X'/Zk_1$ is 
\begin{align*}
&\ll \frac{\tau(k_1)Z}{h} \exp\{-3\sqrt{\log 2Y/h}\}\sum_{n \leq X'/Zk_1} v^{\Omega(n)} \\
&\ll \frac{\tau(k_1)X'}{hk_1}(\log (2\max\{1,hY^2/X'\}))^{v-1}\exp\{-3\sqrt{\log 2Y/h}\}. 
\end{align*}
Here we have used the fact that $ X'/Zk_1 \geq hX'/Z^2=hY^2/X'$, since 
$k_1\leq Z/h$.  Next, on breaking the interval into dyadic intervals we deduce from \eqref{eq:112.2} that 
\begin{align*}
\sum_{Y/k_1<n\leq Z/k_1} \frac{v^{\Omega(n)}}{n}
&\ll \log (X'/Y^2) \max_{H>hY/Z} \sum_{H<n\leq 2H} \frac{v^{\Omega(n)}}{n}\\
&\ll \log (X'/Y^2)(\log (2\max\{1,hY^2/X'\}))^{v-1}, 
\end{align*}
for $v\in [0,1]$. For fixed values of $h$ and $k_1$, it therefore follows from 
\eqref{eq:112.1} that the
contribution from $n > X'/Zk_1$ is 
\begin{align*}
&\ll \frac{\tau(k_1)X'}{hk_1} \exp\{-3\sqrt{\log 2Y/h}\}\sum_{Y/k_1<n \leq Z/k_1} \frac{v^{\Omega(n)}}{n} \\
&\ll \frac{\tau(k_1)X'}{hk_1}\log (X'/Y^2)
(\log (2\max\{1,hY^2/X'\}))^{v-1}\exp\{-3\sqrt{\log 2Y/h}\}. 
\end{align*}
Combining these estimates with partial summation,
we therefore deduce that
\begin{align*}
\sigma_1 &\ll X'(\log\log X') \sum_{h\leq Y}\Big(\frac{ v^{\Omega( h)}}{h} (\log
(Z/h))^2 (\log (2\max\{1,hY^2/X'\}))^{v-1}\\
&\qquad\qquad\qquad\qquad\qquad\qquad\times \exp\{-3\sqrt{\log 2Y/h}\}\Big)\\ 
&\ll X'(\log\log X')^3(\log Y)^{2v-2},
\end{align*}
which is satisfactory for \eqref{majsigma}.

To bound  $\sigma_2$, we estimate trivially the sum over $k_2$ as $\min
\{Z/h, X'/hk_1n\}$. Arguing as above, it follows that 
\begin{align*} 
\sigma_2
&\ll X'\log (X'/Y^2)\sum_{Y<h\leq Z}\frac{ v^{\Omega( h)}}{h}\sum_{ k_1 \leq Z/h }
\frac{(\log Y)^{v-1} }{ k_1}\\&\ll  X'(\log\log X')^3(\log Y)^{2v-2}.
\end{align*}
This therefore completes the proof of \eqref{majsigma}.

The rest of the argument is inspired by the proof of \cite[Theorem 21(ii)]{HT}.
Let $E:=\{\mbox{$p$ prime}: 2<p\leq Y\}$, and introduce the quantities
$$
\Omega(m,E):=\sum_{\colt{p^\nu\parallel m}{p\in E}}\nu,\qquad
E(x):=\sum_{\colt{p \leq x}{p\in E}}\frac{1}{p},
$$ 
for any $m\in \N$ and any $x>0$.  
We will make use of the well-known bound (cf. \cite[Exercise 04]{HT})
\begin{equation}
  \label{eq:classic}
\#\{ m\leq x:  \Omega(m,E)\geq \lambda E(x)\} 
\ll \frac{x}{(\log
x)^{Q(\lambda)}(\log\log x)^{1/2}},
\end{equation}
where $Q$ is given by \eqref{eq:Q}, and which is valid for any $\lambda\in [1,2]$.
We observe that
\begin{equation}
  \label{eq:2411.1}
\sum_{m\in \mcal{B}}|C(m)| \leq  \sum_{1 \leq m\leq X'}
\Big|\sum_{\colt{d\mid m}{Y<d\leq Z}}\chi(d)\Big| ,
\end{equation}
where
$$ 
Y=\frac{{X'}^{1/2}}{(\log X')^{2A+2}}, \quad
Z=\frac{X'}{Y}={X'}^{1/2}(\log X')^{2A+2}.
$$
We will break the sum over $m$ into three parts.

Let $\mcal{B}_1$ denote the set of positive integers $m \leq X'$
such that 
$$
\Omega(m,E)\leq   E(X')/\log 2,
$$
let $\mcal{B}_2$ denote the corresponding set for which 
$$
E(X')/\log 2< \Omega(m,E) \leq  2E(X'), 
$$
and let $\mcal{B}_3$ denote the remaining set of positive integers $m
\leq X'$. 
We will write $S_j=\sum_{m \in \mcal{B}_j}
|\sum_{d}\chi(d)|$, for $1\leq j\leq 3$, with the conditions on $d$ as in \eqref{eq:2411.1}.
We then have
\begin{align*} 
S_1&\leq \sum_{m\in \mcal{B}_1}\sum_{\colt{d\mid m}{Y<d\leq Z}}1
=\sum_{ h+k\leq  E(X')/\log 2}\sum_{\colt{Y<d\leq Z}{\Omega(d,E)=h}}\sum_{\colt{n\leq
X'/d}{\Omega(n,E)=k}}1.
\end{align*}
Since $E(X'/d)=E(X')$ for $d\leq Z$, an application of \cite[Theorem 08]{HT} yields 
$$
\sum_{\colt{n\leq
X'/d}{\Omega(n,E)=k}}1\ll \frac{X'}{d}\exp\{-E(X')\}\frac{E(X')^k}{k!},
$$
uniformly for $k\leq (3-\varepsilon)E(X')$. Hence a  
repeated application of  \cite[Theorem 08]{HT} reveals that
$$
\sum_{\colt{Y<d\leq Z}{\Omega(d,E)=h}}\sum_{\colt{n\leq
X'/d}{\Omega(n,E)=k}}1\ll 
X' \log (Z/Y) \exp\{-2E(X')\}\frac{E(X')^h}{h!}\frac{E(X')^k}{k!},
$$
uniformly for $h,k\leq (3-\varepsilon)E(X')$. It is clear that
$\log (Z/Y)\ll \log\log X'$
and 
\begin{equation}
  \label{eq:0412.1}
E(X')=E(Y)=\log\log Y+O(1) = \log\log X'+O(1).
\end{equation}
Moreover, the binomial theorem implies that 
$$
\ell!\sum_{h+k=\ell}\frac{1}{h!k!}=\sum_{0\leq h\leq \ell}\frac{\ell!}{h!(\ell-h)!}=2^\ell,
$$
for fixed $\ell$.  We therefore deduce 
from \cite[Theorem 09]{HT} that
\begin{align*}
S_1&\ll X' \log\log X' \sum_{ \ell\leq  E(X')/\log 2}
\exp\{-2E(X')\}\frac{(2E(X'))^{\ell}}{\ell!}\\  &\ll X'(\log\log X')^{1/2}\exp\{
-2Q(1/(2\log 2))E(X')\}\\ 
&\ll X'(\log\log X')^{1/2} (\log X')^{-\eta},
\end{align*}
which is satisfactory for the lemma.

We now turn to $S_2$. Let $S_2(\ell)$ denote the overall contribution to $S_2$ from $m$ such
that $\Omega(m,E)=\ell$. There are clearly 
$O(\log\log X')$ possible values for $\ell$. Write $\ell=\lambda E(X')$, for
some $\lambda \in (1/\log 2,2]$. 
Then on combining the Cauchy--Scharwz inequality with \eqref{majsigma}
and \eqref{eq:classic}, we obtain
\begin{align*}
S_2(\ell)^2&\ll \frac{X' }{(\log X')^{Q(\lambda)}(\log\log X')^{1/2} }
\Big((\lambda/2)^{-\lambda E(X')}\sigma (X',\lambda/2)\Big)\\
&\ll  \frac{{X'}^2 (\log\log X')^{5/2} }{
(\log X')^{Q(\lambda) +\lambda(\log (\la/2)-1)+2} },
\end{align*}
since $E(X')=\log\log X'+O(1)$ by \eqref{eq:0412.1}.
Hence it follows that 
$$
S_2=\sum_{\ell\ll \log\log X'} S_2(\ell) \ll  \frac{{X'} (\log\log X')^{9/4} }{
(\log X')^{Q(\lambda)/2 +\lambda(\log (\la/2)-1)/2+1} }. 
$$
This is satisfactory for the statement of the lemma, since
$$
Q(\lambda)/2 +\lambda(\log (\la/2)-1)/2+1\geq Q(1/\log 2),
$$
for $\lambda\geq 1/\log 2$.

It remains to deal with the sum $S_3$, which corresponds to a summation over 
positive integers $m\leq X'$ for which 
$\Omega(m,E)>  2E(X').$
For this we will combine the Cauchy--Schwarz
inequality with \eqref{majsigma} for $v=1$ and the bound
\eqref{eq:classic}, to deduce that
$$
S_3\ll \Big(\frac{X' \sigma(X',1)}{(\log X')^{Q(2)} (\log\log X')^{1/2}}\Big)^{1/2}\ll
\frac{X' (\log\log X')^{5/4}}{(\log X')^{Q(2)/2} }.
$$
This too is satisfactory for the statement of the lemma, since 
$Q(2)/2>\eta$,  and so completes its proof.
\end{proof}

Combining Lemmas \ref{l:2311.4} and \ref{l:2311.5} in
\eqref{eq:2311.3}, we may now conclude that there exists an absolute
constant $c_1>0$ such that 
$$
S_0\ll \frac{L_\infty^\varepsilon  r_\infty r' X^2 (\log\log
  X')^{c_1}}{(\log X')^{\eta}} \ll 
\frac{L_\infty^\varepsilon  r_\infty r' X^2 }{(\log X')^{\eta-\ve}}
\ll
\frac{L_\infty^\varepsilon  r_\infty r' X^2 }{(\log X)^{\eta-\ve}},
$$
since we have assumed that $r'X^{1-\ve}\geq 1$ in the statement of
Theorem \ref{main0}. Once inserted into Lemma \ref{lem31}, this
therefore completes the proof of the theorem.

\section{Linear transformations}

Our proof of Theorems \ref{main1} and \ref{main2} will involve first
establishing the relevant estimate for a specific choice of $j \in
\{*,0,1\}$. The corresponding estimate for the remaining values of $j$
will be obtained via simple changes of variables.
Thus it will be important to consider the effect of
linear transformations on the sums \eqref{eq:Sj},  and that is the
purpose of the present section.

We begin by recording a preliminary result from  group theory.  For any
group $G$ and any subgroup $H \subseteq G$, write $[G:H]$
for the index of $H$ in $G$.

\begin{lem}\label{lem:group}
Let $A, B $ be subgroups of finite index in a group $G$,
such that $[G:A]$ and $[G:B]$ are coprime.  Then we have
$$
[G:A\cap B]=[G:A][G:B].
$$
\end{lem}

\begin{proof}
For any $x,y\in G$ we 
claim that either $xA\cap yB$ is empty, or else it is a left coset
of $A\cap B$ in $G$. Indeed, supposing that 
$xA\cap yB$ is non-empty, we let $c\in 
xA\cap yB$. Note that $xA=cA$ and
$yB=cB$. But then it follows that
$$
xA\cap yB=cA\cap cB=c(A\cap B)
$$
as required.  Thus it follows that the total number of left
cosets of $A\cap B$ in $G$ is
$$
[G:A\cap B]\leq [G:A][G:B].
$$
However, by Lagrange's theorem we have $[G:A\cap B]=[G:A][A:A\cap B]$,
whence $[G:A]$ divides $[G:A\cap B]$. Similarly,
$[G:B]$ divides $[G:A\cap B]$. Thus it follows that
$$
[G:A][G:B]\leq [G:A\cap B],
$$
since $\hcf([G:A],[G:B])=1$.
Once coupled with our upper bound for $[G:A\cap B]$, this completes
the proof of the lemma.
\end{proof}

It will be useful to have a convenient way of referring
back to the statements of our main results.
Let us say that ``Hypothesis-$(j,k)$'' holds if
$S_j(X;\ma{d},\sfg_{\ma{D}})$  satisfies the asymptotic formula
described in Theorem~\ref{main2}  for all
$L_1,\ldots,L_4,\mcal{R}$ that satisfy \textsf{NH}$_k(\ma{d})$. 
Thus Hypothesis-$(j,k)$ amounts to the established existence of an asymptotic
formula 
$$
S_j(X;\ma{d},\sfg_{\ma{D}})=\delta_{j,k}(\A) C_0X^2 +
O\Big(\frac{D^\ve L_\infty^{ \varepsilon}r_\infty r'X^2}{(\log X)^{ \eta-\varepsilon}}\Big),
$$
for $r'X^{1-\ve}\geq 1$, under the assumption that
\textsf{NH}$_k(\ma{d})$ holds. Here
\begin{equation}
   \label{eq:constant}
   C_0=C_{0}(L_1,\ldots,L_4;\ma{d},\sfg_{\ma{D}},\mcal{R}):=
\frac{\pi^4 \meas(\mcal{R})}{\det \sfg_{\ma{D}} } \prod_{p>2}\sigma_p,
\end{equation}
and $\sigma_p$ is given by \eqref{eq:rho0} and \eqref{eq:sig}.

Let $L_1,\ldots,L_4\in\Z[x_1,x_2]$ be binary linear forms, and let
$\mcal{R}\subset \R^2$.  Let $(\ma{d}, \ma{D})\in \mcal{D}$, where $\mcal{D}$
is given by \eqref{eq:D}, and set
\begin{equation}
   \label{eq:X}
   \mcal{X}:=\sfg_{\ma{D}}\cap X\mcal{R}.
\end{equation}
Then for a given matrix
$\M\in \mathrm{GL}_2(\Z)$, 
we define the sum
$$
S_{\M}:=\sum_{\colt{\y\in \Z^2, ~\M\y\in \mcal{X}}{2\nmid y_1, ~y_2\equiv
   j\bmod{2}}}
r\Big(\frac{L_1(\ma{My})}{d_1}\Big)
r\Big(\frac{L_2(\ma{My})}{d_2}\Big)r\Big(\frac{L_3(\ma{My})}{d_3}\Big)
r\Big(\frac{L_4(\ma{My})}{d_4}\Big).
$$
Here, as throughout this paper, we let
$\mathrm{GL}_2(\Z)$ denote the set of non-singular $2\times 2$ integer
valued matrices with non-zero determinant.
Note that $S_{\M}$ depends on $X, \ma{d}, \ma{D}, L_1,\ldots,L_4$
and $j$, in addition to $\ma{M}$. In particular we have
$S_{\M}=S_{j}(X;\ma{d},\sfg_{\ma{D}})$, when
$\M$ is the identity matrix.
In general let us write $\|\M\|$ to denote the maximum modulus of the
coefficients of $\M$. Bearing all this notation in mind, the following elementary result will
prove  useful.

\begin{lem}\label{lem:linear}
Let $(j,k) \in \{*,0,1\}\times \{0,1,2\}$ and suppose
Hypothesis-$(j,k)$ holds.
Let $\M\in \mathrm{GL}_2(\Z) $ such that $\det \ma{M}=2^m$ for some $m
\in \Z_{\geq 0}$,
  and define
$M_i(\y):=L_i(\ma{My})$. Let $\ve>0$ and suppose that 
$r'(L_1,\ldots,L_4,\mcal{R})X^{1-\ve}\geq 1$. 
Assume that $M_1,\ldots,M_4,\mcal{R}$ satisfy \textsf{NH}$_k(\ma{d})$. 
Then we have
$$ 
S_{\M}=\frac{\delta_{j,k}(\A\M) C_0}{\det \M} X^2 +
O\Big(\frac{D^\ve L_\infty^{ \varepsilon} \|\M\|^{ \varepsilon}r_\infty(\mcal{R}_\ma{M}) r'
 X^2}{(\log X)^{ \eta-\varepsilon}}\Big),
$$
where $D=D_1 \cdots D_4$, 
$L_\infty=L_\infty(L_1,\ldots,L_4)$,  $r'=r'(L_1,\ldots,L_4,\mcal{R})$, 
and
\begin{equation}
  \label{eq:2311.6}
\mcal{R}_{\M}:=\{\M^{-1}\z: \z\in \mcal{R}\}.
\end{equation}

\end{lem}

It is important to note that the definition of $\sigma_p$ that appears
in \eqref{eq:constant} is precisely as in
\eqref{eq:sig}. Thus it involves lattices that depend
on $L_1,\ldots,L_4$, rather than $M_1,\ldots, M_4$. The net outcome
  of Lemma \ref{lem:linear} is that
for linear transformations that preserve the relevant normalisation
conditions and have determinant $2^m$ for some $m \geq 0$, the main term of the corresponding asymptotic formula
should be multiplied by 
$
\delta_{j,k}(\A\M)(\delta_{j,k}(\A)\det \M)^{-1}.
$

\begin{proof}[Proof of Lemma \ref{lem:linear}]
Recall the definition \eqref{eq:X} of $\mcal{X}$, and the notation
introduced in \eqref{eq:lattice}.
  We begin by noting
that $\M\y \in \mcal{X}$ if and only if
$\y \in \sfl_{\M}\cap \mcal{R}_{\M}$, where
$$
\sfl_{\M}:=\{\y\in \Z^2: D_i \mid L_i(\ma{My})\} =
\sfg(\ma{D};M_1,\ldots,M_4),
$$
and $\mcal{R}_{\M}$ is given by \eqref{eq:2311.6}.
Moreover, $M_1,\ldots,M_4, \mcal{R}_{\M}$ will satisfy
\textsf{NH}$_k(\ma{d})$ if $M_1,\ldots,M_4, \mcal{R}$ do. 
We claim that
\begin{equation}
   \label{eq:claim}
   \det \sfl_\ma{M}=\det \sfg(\ma{D};M_1,\ldots,M_4)=\det \sfg(\ma{D};L_1,
\ldots,L_4),
\end{equation}
for any matrix $\ma{M}\in \mathrm{GL}_2(\Z)$ such that
$\hcf(\det\ma{M}, D)=1$. In particular, since $\ma{M}$ has 
determinant $2^m$ for some $m\in \Z_{\geq 0}$, this holds for any
$\ma{D}\in\N^4$ such that $2\nmid D$. Assume \eqref{eq:claim}  to 
be true for the moment, and note that
$$
\meas(\mcal{R}_{\M})=\frac{\meas(\mcal{R})}{\det\M},\qquad
r'(M_1,\ldots,M_4, \mcal{R}_\M)=r'( L_1,\ldots,L_4,\mcal{R})=r',
$$
in the notation of \eqref{eq:r'}. Recalling the definitions in
\eqref{eq:Linf} and \eqref{eq:rinf}, 
we therefore deduce from Hypothesis-$(j,k)$ that
\begin{align*} 
S_{\M}
=&\frac{\delta_{j,k}(\A\M) 
\pi^4 \meas(\mcal{R}_{\ma{M}})}{\det
\sfg(\ma{D};M_1,\ldots,M_4) )}X^2 \prod_{p>2}\sigma_p'\\ &\quad +
O\Big(D^\ve L_\infty(M_1,\ldots,M_4)^{ \varepsilon}r_\infty(\mcal{R}_\M) r' 
\frac{X^2}{(\log X)^{\eta-\varepsilon}}\Big)\\ 
=&\frac{\delta_{j,k}(\A\M) 
\pi^4 \meas(\mcal{R})}{(\det \ma{M})(\det \sfg(\ma{D};L_1,\ldots,L_4))
)}X^2 \prod_{p>2}\sigma_p' \\
&\quad +
O\Big( D^\ve L_\infty(M_1,\ldots,M_4)^{ \varepsilon}r_\infty(\mcal{R}_\M) r' 
\frac{X^2}{(\log X)^{\eta-\varepsilon}}\Big),
\end{align*}
where
$$
\sigma_p'=
\Big(1-\frac{\chi(p)}{p}\Big)^4
\sum_{a,b,c,d=0}^{\infty}\chi(p)^{a+b+c+d}\rho_0(p^a,p^b,
p^c,p^d;\ma{D};M_1,\ldots,M_4)^{-1}.
$$
On noting that $L_\infty(M_1,\ldots,M_4) \leq
L_\infty(L_1,\ldots,L_4)\|\M\|$, we see that the error term in this estimate for $S_\M$ is 
as claimed in the statement of the lemma. 
Moreover, \eqref{eq:rho0} and \eqref{eq:claim} give 
\begin{align*}
\rho_0(\ma{h};\ma{D};M_1,\ldots,M_4)
&=\frac{\det
   \sfg\big(([D_1,d_1h_1],\ldots,[D_4,d_4h_4]);M_1,\ldots,M_4\big)}{\det
\sfg(\ma{D};M_1,\ldots,M_4)} \\ &=\rho_0(\ma{h};\ma{D};L_1,\ldots,L_4),
  \end{align*}
for any $\ma{h}\in \N^4$ such that $2\nmid h_1\cdots h_4$. 
Hence $\sigma_p'= \sigma_p$.

In order to complete the proof of Lemma \ref{lem:linear} it remains to
establish \eqref{eq:claim}.  For any matrix $\ma{N}\in
\mathrm{GL}_2(\Z)$ and any lattice $\sfl\subseteq \Z^2$, it is easily checked
that
$$\det (\ma{N}\sfl) = \det \ma{N} \det \sfl,
$$
where $\ma{N}\sfl:=\{\ma{N}\x: \x \in \sfl\}$. It therefore follows
that
$$
\det \sfl_{\ma{M}}= \frac{\det (\ma{M}\sfl_\ma{M})}{\det \M}.
$$
Note that $\ma{M}\sfl_{\ma{M}}=\mathsf{M}\cap
\sfg(\ma{D};L_1,\ldots,L_4)$,
where $\mathsf{M}=\{\M\y: \y\in\Z^2\}$. In particular we have $\det
\mathsf{M}=\det \M$.  To establish \eqref{eq:claim}, it therefore suffices  to
show that
$$
\det (\mathsf{L}\cap
\sfg(\ma{D};L_1,\ldots,L_4))=(\det \mathsf{L}) (\det \sfg(\ma{D};L_1,\ldots,L_4))
$$
for any lattice $\mathsf{L} \subseteq \Z^2$ such that
$\hcf(\det\mathsf{L}, D_1D_2D_3D_4)=1$.
But this follows immediately from Lemma \ref{lem:group}, since the
determinant of a sublattice of $\Z^2$ is equal to its index
in $\Z^2$.
\end{proof}

\section{Proof of Theorem \ref{main1}}\label{lattices}

We are now ready establish the statement of Theorem
\ref{main1}. The proof will be in two stages: first we will establish the result for
$j=*$, and then we will proceed to handle the cases $j\in \{0,1\}$.
Our proof of the estimate for $j=*$ is actually a straightforward
generalisation of an argument already present in Heath-Brown's work
\cite[\S 7]{h-b03}, but we will include full details here for
the sake of completeness.

Assume that $(\ma{d},\ma{D})\in \mcal{D}$, where
$\mcal{D}$ is given by \eqref{eq:D}. In particular it follows that
there exists $\x\in\sfg_{\ma{D}}$ such that $x_1\equiv 1 \bmod{4}$,
where $\sfg_{\ma{D}}$ is given by \eqref{eq:lattice}.
Indeed, the vector $\x=D_1^2D_2^2D_3^2D_4^2(1,1)$ is clearly 
satisfactory. 
In estimating $S_{*}(X;\ma{d},\sfg_{\ma{D}})$, our goal is to replace the summation
over lattice points $\x \in \sfg_{\ma{D}}$ by a summation over all
integer points restricted to a certain region.  
Given any basis $\ma{e}_1, \ma{e}_2$ for $\sfg_{\ma{D}}$, let 
$M_i(\v)$ be the linear form obtained from
$d_i^{-1}L_i(\x)$ via the change of variables $\x\mapsto v_1\ma{e}_1+v_2\ma{e}_2$.
We claim that there is a choice of basis such that
\begin{equation}\lab{20-trans'}
M_i(\v)\eqm{v_1}{4},
\end{equation}
for each $i$, and also
\begin{equation}
  \label{eq:3011.3}
  \|\M\|\ll \det \sfg_{\ma{D}},
\end{equation}
where $\M$ denotes the matrix formed from the basis vectors
$\ma{e}_1,\ma{e}_2$.  To check the claim we let $\ma{e}_1, \ma{e}_2$
be a minimal basis for $\sfg_{\ma{D}}$.
%, in the sense of
%\cite[\S 2]{n-2}. 
Thus we may assume that 
\begin{equation}
  \label{eq:3011.2}
|\ma{e}_1||\ma{e}_2| \ll \det \sfg_{\ma{D}}.
\end{equation}
Now there must exist integers $w_1,w_2$ such that
$
w_1 e_{11}+ w_2 e_{21} \equiv 1 \bmod{4},
$
since we have seen that there exists $\x\in\sfg_{\ma{D}}$ such that $x_1\equiv 1 \bmod{4}$.
In particular we may assume without loss of generality that $e_{11}$ is
odd, and after multiplying $\ma{e}_1$ by $\pm 1$, we may as well
assume that $e_{11}\equiv 1 \bmod{4}$.  Next, on replacing $\ma{e}_2$ by
$\ma{e}_2-k\ma{e}_1$ for a suitable integer $k\in \{0,1,2,3\}$, we may further assume
that $4\mid e_{21}$.   In view of \eqref{eq:3011.2}, this basis certainly satisfies \eqref{eq:3011.3}.
Moreover, the normalisation conditions on $L_1,\ldots,L_4$ imply that
$$
d_iM_i(\v)= L_i(v_1\ma{e}_1+v_2 \ma{e}_2)
\equiv d_i(v_1e_{11}+v_2 e_{21}) \equiv d_iv_1 \pmod{4},
$$
which therefore establishes \eqref{20-trans'} since each $d_i$ is odd.

Note that we must sum only over odd values of $v_1$, since we 
have been summing over odd $x_1$ in 
$S_{*}(X;\ma{d},\sfg_{\ma{D}})$. 
On recalling the definition \eqref{eq:2311.6} of  $\mcal{R}_{\ma{M}}$,  
we may therefore deduce that
\begin{align*}
S_{*}(X;\ma{d},\sfg_{\ma{D}})
&=\sum_{\colt{\v \in \Z^2\cap X\mcal{R}_{\ma{M}}}{2\nmid v_1}}
r\big(M_1(\v)\big)\cdots r\big(M_4(\v)\big).
\end{align*}
Note that \eqref{20-trans'} holds by construction, and also $M_i(\v)>0$
  for every $\v$ in the summations.
We are therefore in a position to apply Theorem \ref{main0} to
estimate this quantity. 
In view of \eqref{eq:3011.3} and the fact that $\det \sfg_{\ma{D}}\mid
D=D_1\cdots D_4$, we may deduce that
$$
L_\infty(M_1,\ldots,M_4)\leq \|\ma{M}\| L_\infty(L_1,\ldots,L_4) \ll
D L_\infty, 
$$
where $L_\infty=L_\infty(L_1,\ldots,L_4)$, as 
usual. Next we deduce from \eqref{eq:3011.3} that
$$
r_\infty(\mcal{R}_{\ma{M}})\leq \frac{\|\ma{M}\|}{|\det \ma{M}|} r_\infty(\mcal{R}) \ll
r_\infty(\mcal{R})=r_\infty,
$$
since $|\det \ma{M}|=\det \sfg_{\ma{D}}$, and furthermore
$$
r'(M_1,\ldots,M_4,\mcal{R}_{\ma{M}})=r'(L_1,\ldots,L_4,\mcal{R})=r'.
$$
Moreover, it is clear that $\meas(\mcal{R}_{\ma{M}})=\meas(\mcal{R})/|\det
\ma{M}|$. It therefore follows
from Theorem \ref{main0} that
$$
S_{*}(X;\ma{d},\sfg_{\ma{D}})
= \frac{4\pi^4 \meas(\mcal{R})}{\det \sfg_{\ma{D}}}X^2 \prod_{p>2}\sigma_p^* +
O\Big(\frac{D^\ve L_\infty^{ \varepsilon}r_\infty r'X^2}{(\log X)^{ \eta-\varepsilon}}\Big),
$$
where $\sigma_p^*$ is given by \eqref{eq:sig*}, but with $\rho_*(\ma{h})=
\det\sfg(\ma{h};M_1,\ldots,M_4)$. To calculate this quantity we note
that it is just  the index of
$$
\sfl_1=\{\x=v_1\ma{e}_1+v_2\ma{e}: \v\in\Z^2, h_i\mid M_i(\v)\}
$$
in
$
\sfl_2=\{\x=v_1\ma{e}_1+v_2\ma{e}: \v\in\Z^2\},
$
whence
\begin{align*}
\rho_*(\ma{h})=[\sfl_1:\sfl_2]=
\frac{\det\sfl_1}{\det \sfl_2}
&=\frac{\det \{\x\in\sfg(\ma{D};L_1\ldots,L_4): d_ih_i \mid
L_i(\x)\}}{\det \sfg(\ma{D};L_1,\ldots,L_4)}\\
&=\rho_0(\ma{h};\ma{D};L_1,\ldots,L_4),
\end{align*}
in the notation of \eqref{eq:rho0}.  This therefore establishes the
estimate in Theorem~\ref{main1} when $j=*$.

In order to complete the proof of Theorem \ref{main1} it remains
  to handle the cases $j=0,1$. For this we carry out the change of
  variables $\x=\M\y$, with 
$$
\M=\Big(
\begin{array}{cc}
1&0\\
j &2
\end{array}
\Big).
$$
This has the effect of transforming the sum into one over integers
$\y$ such that $y_1$ is odd, without any restriction on $y_2$.
Moreover, it is clear that 
$L_i(\M\y)=L_i(y_1,jy_1+2y_2)\equiv d_i y_1 \bmod{4}$, 
so that together with $\mcal{R}$, the new linear forms satisfy 
\textsf{NH}$_0(\ma{d})$. 
Since we have already seen that Hypothesis-$(*,0)$ holds, we may
deduce from Lemma~\ref{lem:linear} that
$$
S_{j}(X;\ma{d},\sfg_{\ma{D}})=
\frac{\del_{*,0}(\A\M)C_0 }{2}X^2
+ 
O\Big( 
\frac{ D^\ve L_\infty^{ \varepsilon} r_\infty   r'X^2}{(\log
X)^{\eta-\varepsilon}}\Big),  
$$
for $j=0,1$, where $C_0$ is given by \eqref{eq:constant}.
The statement of Theorem \ref{main1} follows 
since $\del_{*,0}(\A\M)=\del_*=4$, by \eqref{eq:C}.

\section{Proof of Theorem \ref{main2}}\label{s:t2}

We are now ready to establish Theorem \ref{main2}.
 Let $(j,k)\in
\{*,1,2\}\times \{1,2\}$ and let $(\ma{d},\ma{D})\in \mcal{D}$.
It will ease notation if we write
$S_{j,k}(X)$ to denote the sum $S_{j}(X;\ma{d},\sfg_{\ma{D}})$, when
$L_1,\ldots,L_4,\mcal{R}$ are assumed to satisfy \textsf{NH}$_k(\ma{d})$. 
Furthermore, let us write
\begin{equation}
   \label{eq:sab}
\mcal{S}_{\al}:=\{\y\in \Z^2: ~y_1 \equiv  1 \bmod{4}, ~y_2
\equiv \al \bmod{2}\},
\end{equation}
for $\al \in \{*,0,1 \}$.
We begin by showing how an estimate for $k=1$ can be used to
deduce a corresponding estimate for the case $k=2$.  

Suppose that $k=2$ and $j=1$.  We may clearly assume that the summation in
$S_{1,2}(X)$ is only over values of $x_1\equiv 1
\bmod{4}$ and $x_2 \equiv d_2 
\bmod{4}$, since the summand vanishes unless
$$ 
d_1x_1 \equiv 2^{-k_1}L_1(\x) \equiv d_1  \pmod{4},\quad
x_2 \equiv  2^{-k_2}L_2(\x) \equiv d_2 \pmod{4}.
$$ 
Write $\kappa=\pm 1$ for the residue modulo  $4$ of $d_2$, and 
choose an integer $c$ such that
$$
a_j+ b_j(\kappa+4c) \neq 0,
$$
for $j=3,4$, where $a_j,b_j$ are as in \eqref{L3L4}.
This is plainly always possible with $c\in \{0,1,2\}$.
We will carry out the transformation $\x=\M_{c,d_2}\y$, with $\M_{c,d_2}$ given by \eqref{defM}.
Such a transformation is valid if and only if there exists an integer
$y_2$ such that 
$x_2-(\kappa+4c)x_1=4y_2$ where $\kappa  \equiv d_2 \bmod 4$. Thus the
transformation is certainly valid for 
$x_1\equiv 1 \bmod{4}$ and $x_2 \equiv d_2  \bmod{4}$, bringing
the linear forms into new forms
$M_i(\y)=L_i(\M_{c,d_2}\y)$, say. It is not hard to see that
$M_1,\ldots,M_4,\mcal{R}$ will satisfy \textsf{NH}$_1(\ma{d})$.
 There is now no $2$-adic restriction on $y_2$,
so that the summation is over $\y\in \SS_{*}$, in the notation of
\eqref{eq:sab}. We clearly have  $r_\infty(\mcal{R}_{\M_{c,d_2}})\ll r_\infty(\mcal{R} ).$
  By combining Lemma \ref{lem:linear} with the assumption that Hypothesis-$(*,1)$
holds, we therefore obtain 
$$
S_{1,2}(X)
=\frac{\delta_{*,1}(\A\M_{c,d_2})C_0}{4}X^2 
+O\Big( 
\frac{ D^\ve L_\infty^{ \varepsilon} r_\infty  r'X^2}{(\log
X)^{\eta-\varepsilon}}\Big),
$$
where $C_0$ is given by \eqref{eq:constant}.  This is clearly
satisfactory for the statement of Theorem \ref{main2}, since \eqref{eq:1811.2} yields
$\delta_{1,2}(\A)=\delta_{*,1}(\A\M_{c,d_2})/4$. 

To handle $S_{0,2}(X)$ we will need to extract $2$-adic powers from
the variable $x_2$. Accordingly, we write $x_1=y_1$ and $x_2=2^\xi
y_2$, 
for $\xi\geq 1$ and $y_2\equiv 1 \bmod{2}$. This corresponds to the transformation
$\x=\M_\xi\y$ with $\M_\xi$ given by \eqref{defMxi}.
The resulting linear forms $M_i(\y)=L_i(\M_\xi\y)$  will continue to satisfy
\textsf{NH}$_2(\ma{d})$, 
 and the summation will be over $\y \in \SS_1$.  Moreover, the
restriction $\x \in X\mcal{R}$ in the definition of $S_{0,2}(X)$ forces the upper bound
$\xi \leq  \log(r_\infty X)$.  It turns that this is too crude for our purposes and we
must work a little harder to control the contribution from large values of $\xi$.
Recall the definitions \eqref{eq:Linf}, \eqref{eq:rinf} of $L_\infty$
and $r_\infty$. We will show that 
\begin{equation}\label{queue}
\sum_{\colt{\y \in \Z^2}{\M_\xi \y \in \mcal{X}}}
r\Big(\frac{L_1(\ma{M}_\xi\ma{y} )}{d_1}\Big)
r\Big(\frac{L_4(\ma{M}_\xi\ma{y})}{d_4}\Big)
\ll  (D2^{\xi}L_\infty)^{\varepsilon} 
 \Big(r_\infty^2
\frac{ X^2}{2^{\xi }}
+  r_\infty^{1+\ve} X^{1+\varepsilon }\Big).
\end{equation}
Define the multiplicative function $r_1$ via
$$
r_1(p^\nu)=\left\{
\begin{array}{ll}
1+\nu, &\mbox{if $p\mid d_1d_2d_3d_4$,}\\
r(p^\nu), &\mbox{if $p\nmid d_1d_2d_3d_4$,}
\end{array}
\right.
$$
for any prime power $p^\nu$.
Then we have
$$
r\Big(\frac{L_1(\ma{M}_\xi\ma{y} )}{d_1}\Big)
\cdots
r\Big(\frac{L_4(\ma{M}_\xi\ma{y})}{d_4}\Big) 
\leq r_1(F(\ma{y})), 
$$
where $F(\ma{y})=L_1(\ma{M}_\xi\ma{y} ) \cdots L_4(\ma{M}_\xi\y)$. 
The maximum modulus of the coefficients of this binary form is 
$O(L_\infty^4 2^{4\xi}).$ Hence \eqref{queue} follows easily on taking $X_1=r_\infty X$ and
$X_2=2^{-\xi}r_\infty X$ in \cite[Corollary 1]{nair}.
Note that it would not be sufficient to work instead with the trivial upper bound 
$O(L_\infty^\varepsilon r_\infty^{2+\varepsilon }2^{-\xi}X^{2+\varepsilon})$.

To complete our estimate for $S_{0,2}(X)$ we will combine Lemma \ref{lem:linear} with Hypothesis-$(1,2)$
to handle the contribution from $\xi\leq \xi_1$, 
and we will use \eqref{queue} to handle the contribution from $\xi_1 
 <\xi \leq \log (r_\infty X)$, for a value of $\xi_1$  to be  determined.  
We claim that
\begin{equation}\label{eq:2411.2}
r_\infty\leq 2 L_\infty r'.
\end{equation}
To see this, suppose that $\z\in \RR$ is such that
$r_\infty=|z_1|$, say. Then it follows that
$$
r_\infty \leq |a_3b_4-a_4b_3||z_1| = |b_4L_3(\z)-b_3L_4(\z)| \leq
2L_{\infty} r',
$$
in the notation of \eqref{L3L4}. Write 
$$
E_1=\frac{  2^{ \varepsilon \xi} X^2}{(\log
X)^{\eta-\varepsilon}}, \quad
E_2= L_\infty 2^{-\xi+\varepsilon \xi}X^2+
{r'}^{\ve}2^{\varepsilon\xi}X^{1+\varepsilon },
$$  
and choose $\xi_1\in\N$ such that   
$2^{\xi_1-1}<L_\infty (\log X)^{\eta}\leq 2^{\xi_1}$. 
Next we note that 
$$
C_0\ll D^\ve\frac{r_\infty^2}{\det \sfg_{\ma{D}}}\ll
D^\ve L_\infty r_\infty r',
$$
in \eqref{eq:constant}. Hence we deduce from \eqref{eq:1811.1} and \eqref{eq:2411.2} that
\begin{align*}
S_{0,2}(X)
=&\sum_{\xi =1}^{\xi_1} \frac{\delta_{1,2}(\A\M_\xi) C_0 }{2^{\xi}}X^2 
+ 
O\Big(D^\ve L_\infty^{ \varepsilon} r_\infty r'
\big(\sum_{\xi=1}^{\xi_1} 
 E_1 +
\hspace{-0.2cm}
\sum_{\xi=\xi_1+1}^{\log(r_\infty X)}
 E_2 \big)\Big)\\
\\
=&\sum_{\xi=1}^{\infty} \frac{\delta_{1,2}(\A\M_\xi) C_0 }{2^{\xi}}X^2 +
O\Big(\frac{D^\ve  L_\infty^{\varepsilon} r_\infty r' X^2}{(\log
X)^{\eta-\varepsilon}}\Big)\\
=&\delta_{0,2}(\A) C_0 X^2+ 
O\Big(\frac{D^\ve  L_\infty^{\varepsilon} r_\infty r' X^2}{(\log
X)^{\eta-\varepsilon}}\Big).
\end{align*}
This completes the treatment of $S_{0,2}(X)$.

The estimate for $S_{*,2}(X)=S_{0,2}(X)+S_{1,2}(X)$ 
is now an immediate consequence of our
estimates for $S_{0,2}(X)$ and $S_{1,2}(X)$.
Indeed we plainly have
\begin{align*}
\delta_{*,2}(\A)=\delta_{0,2}(\A)+\delta_{1,2}(\A)
=\sum_{\xi=0}^{\infty} \frac{\delta_{1,2}(\A\M_\xi) }{2^{\xi}}.
\end{align*}
The argument that we have presented here makes crucial use of our
previous work \cite{nair} to control the contribution from large
values of $\xi$ that feature in the change of variables. This
basic technique will recur at several points in the proof of Theorem
\ref{main2}.   Rather than repeating the exact same details
each time, however, we will merely refer the reader back to \eqref{queue} in
order to draw attention to this basic chain of reasoning.

Let $j\in \{*,0,1\}$.
It remains to estimate  $S_{j,1}(X)$. In fact it will suffice to deal only with the case
$j=1$. Indeed, the remaining cases are  handled just as above,
leading to \eqref{eq:1811.1} in the case $k=1$.
Assume that $L_1,\ldots,L_4,\mcal{R}$ satisfy \textsf{NH}$_1(\ma{d})$. 
We have
$$
S_{1,1}(X)=\sum_{\x\in \SS_1\cap \mcal{X}}
r\Big(\frac{L_1(\x)}{d_1}\Big)
r\Big(\frac{L_2(\x)}{d_2}\Big)r\Big(\frac{L_3(\x)}{d_3}\Big)
r\Big(\frac{L_4(\x)}{d_4}\Big),
$$
where $\SS_1$ is given by \eqref{eq:sab} and $\mcal{X}=\sfg_{\ma{D}}\cap X\mcal{R}$.
Let us write $S(X)=S_{1,1}(X)$ for short. Our aim is to find
a linear change of variables
$\x=\M\y,$ for some $\M\in \mathrm{GL}_2(\Z)$, 
taking the linear forms $L_i$ into forms
$M_i(\y)=L_i(\M\y)$ such that
\begin{equation}
   \label{eq:necc}
2^{-\ell_i}M_i(\y)\equiv d_i y_1 \pmod{4},
\end{equation}
for certain  $\ell_i\in \Z_{\geq 0}$.
On setting $M_i'=2^{-\ell_i}M_i$, so that $M_1',\ldots,M_4'$
satisfy \textsf{NH}$_0(\ma{d})$, 
we will then be in a position to apply Lemma
\ref{lem:linear} under the assumption that Hypothesis-$(j,0)$ holds for
$j\in \{*,0,1\}$. Indeed, we have already seen that Theorem
\ref{main1} holds in the previous section.

Let $\x \in \SS_1\cap \mcal{X}$, so that $x_1 \equiv 1 \bmod{4}$ and $2\nmid
x_2$. Recall the assumption that \eqref{eq:L34} holds for appropriate
$k_j,a_j',b_j',\mu_j,\nu_j$. At certain points of the argument we will find it convenient to extract
$2$-adic factors from the terms $2^{-k_j}L_j(\x)$. Let us write
\begin{equation}
  \label{eq:extract2}
{\xi_j}=\nu_2\big(2^{-k_j}L_j(\x)\big), 
\end{equation}
for $j=3,4$. This will allow certain linear transformations to take
place, and it turns  
out that the matrices needed to bring $L_i$ in line with
\eqref{eq:necc} will all take the shape
\begin{equation}
  \label{eq:base-matrix}
\M=\Big(
\begin{array}{cc}
1&0\\
A &2^{\xi+2}
\end{array}
\Big),
\end{equation}
for appropriate non-negative integers  $A\in [0,2^{\xi+2})$ 
 and $\xi$.  Here $\xi$ will be a simple function of $\xi_3$ and $\xi_4$.
Assuming that we are now in a position to combine 
Lemma \ref{lem:linear} with Hypothesis-$(j,0)$, 
we will then obtain a contribution
\begin{equation}\label{eq:e-game}
\begin{split}
&=\frac{ \delta_{j,0}(\A\M) C_0 }{2^{\xi+2}} X^2  
+O\Big(\frac{D^\ve  L_\infty^{ \varepsilon} r_\infty r' 2^{\xi\varepsilon}X^2}{(\log
X)^{\eta- \varepsilon}}\Big)\\
&=\frac{ \delta_{j}C_0 }{2^{\xi+2}} X^2
+O\Big(\frac{D^\ve  L_\infty^{ \varepsilon} r_\infty r'2^{\xi\varepsilon} X^2}{(\log
X)^{\eta- \varepsilon}}\Big),
\end{split}
\end{equation}
since \eqref{eq:C} implies that $\delta_{j,0}(\ma{B})=\delta_j$, and furthermore,
$$
r_\infty(\mcal{R}_\M) \leq 
\frac{\|\ma{M}\|}{\det \ma{M}}
r_\infty(\mcal{R}) = r_\infty(\mcal{R})=r_\infty.
$$ 
Finally, we will need to sum this quantity
over all available $\xi_3,\xi_4$. It is here that we must return to
\eqref{queue} and repeat the sort of argument used there to handle
the large values of $\xi_3$ and $\xi_4$.

Under any transformation $\x=\M\y$, with $\M$ taking the shape
\eqref{eq:base-matrix}, it follows from condition (iv$'$)$_{\ma{d}}$ in the introduction that
$$
2^{-k_j}L_j(\M\y)\equiv d_jy_1 \pmod{4}
$$
for $j=1,2$. As long as our transformations have this general shape
therefore, we will be able to focus our attention on the effect that
the transformation has on the linear forms $L_3,L_4.$
Unfortunately, bringing these forms into the required shape isn't entirely
straightforward, and the permissible choice of $\M$ depends
intimately upon the values of $a_j',b_j',\mu_j,\nu_j$ in \eqref{eq:L34}.
We may assume that these constants satisfy
\eqref{eq:aibi} and \eqref{eq:munu}, and we proceed to consider a
number of distinct subcases separately.

\subsection{The case $\max\{\mu_3,\nu_3\}\geq 1$ and $\max\{\mu_{4},\nu_4\}\geq 1$} 

This case is equivalent to the case in which precisely two of the
exponents $\mu_3,\mu_4,\nu_3,\nu_4$ are non-zero, which in turn is
equivalent to the statement that $\mu_j+\nu_j\geq 1$ for $j=3,4$, since $\mu_3\nu_3=\mu_4\nu_4=0$.
In particular it follows that $2^{-k_j}L_j(\x)$ is odd for any odd values of $x_1,x_2$.
Recall that  the
summation is over $x_1\equiv  1 \bmod{4}$  
and $x_2$ odd in $S(X)$.
Let us write $g$ for the number of values of $\gamma \in \{-1,1\}$
such that
\begin{equation}
  \label{eq:1611.1}
2^{-k_j}L_j(1,\gamma) =2^{\mu_j}a_j' +  2^{\nu_j}b_j'\gamma \equiv d_j \pmod{4} 
\end{equation} 
for $j=3$ and $4$. Our aim is to show that 
\begin{equation}\label{eq:d11}
\del_{1,1}(\A)=g,
\end{equation}
which we claim is satisfactory for \eqref{delta1}--\eqref{delta3}.
To see this, we suppose first that $\nu_3,\nu_4\geq 1$. Then it is clear that 
$g=2$ if $a_j'$ is congruent to $ d_j-2^{\nu_j}$ modulo $4$ for $j=3,4$, and $g=0$
otherwise.  When $\mu_3,\mu_4\geq1$, we have $g=1$ if
$b_3'd_3-2^{\mu_3}\equiv
b_4'd_4-2^{\mu_4} \bmod 4$, and $g=0$ otherwise.
When $\mu_4,\nu_3\geq 1$ we have $g=1$ when 
$a_3'\equiv d_3-2^{\nu_3} \bmod{4}$, the value of $\gamma$ being given by
the residue of $b_4'd_4-2^{\mu_4}$ modulo $4$, and $g=0$ otherwise.
Finally, the case $\mu_3,\nu_4\geq 1$ is symmetric.

It remains to establish \eqref{eq:d11}. 
We may clearly proceed under the assumption that $g\geq 1$.
Let us write $S(X)=\sum_{\gamma} S(X;\gamma)$, where $S(X;\gamma)$ is the overall
contribution to $S(X)$ from vectors such that $x_2 \equiv \gamma  
 \bmod{4}$,
and the summation is over the $g$ values of $\gamma$ for which \eqref{eq:1611.1} holds.
We will carry out the transformation
$$
\M=\Big(
\begin{array}{cc}
1&0\\
\gamma &4
\end{array}
\Big).
$$
This transformation is valid if and only if there exists an integer
$y_2$ such that $x_2=\gamma y_1+4y_2$, for each $\x$ in $S(X)$.  This is
clearly true for $x_1=y_1 \equiv  1\bmod{4}$ and $x_2 \equiv \gamma  \bmod{4}$.
Next we observe that \eqref{eq:necc} holds for the
new linear forms $M_i(\y)=L_i(\M\y)$, since \eqref{eq:1611.1} holds
for $j=3,4$. The summation over $\y$ is now over $\y\in
\mcal{S}_{*}$, since as usual the condition $y_1 \equiv 1 \bmod{4}$ 
is
automatic for odd values of $y_1$ such that $r(M_1(\y)/d_1)\neq 0$.
In line with \eqref{eq:e-game}, we therefore deduce from Lemma
\ref{lem:linear} and Hypothesis-$(*,0)$ that 
\begin{align*}
S(X;\gamma)
&=
\frac{ \delta_{*}C_0 }{4} X^2  +O\Big(\frac{D^\ve  L_\infty^{\varepsilon} r_\infty r' X^2}{(\log
X)^{\eta- \varepsilon}}\Big)
=C_0 X^2 +O\Big(\frac{D^\ve  L_\infty^{ \varepsilon} r_\infty r' X^2}{(\log
X)^{\eta- \varepsilon}}\Big), 
\end{align*}
when $\gamma$ is admissible.  We complete the proof of \eqref{eq:d11} by summing over
the $g$ admissible  choices for $\gamma$.

\subsection{The case $\mu_3=\mu_4=0$ and $\max\{\nu_3,\nu_4\}\geq 1>
\min\{\nu_3,\nu_4\}=0$}

For reasons of symmetry we may restrict ourselves to the case 
$\nu_3\geq 1$ and $\nu_4=0$. For $\x \in \SS_1\cap \mcal{X}$ the term
$2^{-k_3}L_3(\x)$ is odd, whereas $2^{-k_4}L_4(\x)$  is always
even. We note that $r(L_3(\x)/d_3)$  is non-zero if and only if
$a_3'\equiv d_3-2^{\nu_3} \bmod 4$.  We must show  that \eqref{delta4}
holds with $(j_1,j_2)=(4,3)$.

Let us write $\xi_4=\nu_2(2^{-k_4}L_4(\x))$, as in
\eqref{eq:extract2}. Then necessarily $\xi_4\geq1$, since 
$\x\in \SS_1$. We now see that in order for $r(2^{-k_4-\xi_4}L_4(\x)/d_4)$
to be non-zero, it is necessary and sufficient that 
\begin{equation}
  \label{eq:1711.1}
x_2\equiv(d_42^{\xi_4}-a'_4x_1)\overline{b_4'} 
\equiv(d_42^{\xi_4}-a'_4)\overline{b_4'}x_1  \pmod {2^{{\xi_4}+2}},
\end{equation}
where $\overline{b_4'}$ is the multiplicative inverse of $b_4'$ modulo $2^{{\xi_4}+2}$.
Here, we have used that the fact $x_1\equiv  1 \bmod{4}$ in the
summation over $\x$. For each  ${\xi_4}\geq 1 $ 
we make the transformation
\begin{equation}
  \label{eq:0112.1}
\M=\Big(
\begin{array}{cc}
1&0\\
A &2^{{\xi_4}+2}
\end{array}
\Big),
\end{equation}
where $A\in [0,2^{{\xi_4}+2}) $ denotes the residue of
$(d_42^{\xi_4}-a'_4)\overline{b_4'}$ modulo $2^{{\xi_4}+2}$. This
brings $L_3, L_4$ into a satisfactory shape for
\textsf{NH}$_0(\ma{d})$, by which we mean that
$2^{-k_3}L_3(\M\y) \equiv d_3y_1 \bmod{4}$ and 
$2^{-k_4-\xi_4}L_4(\M\y) \equiv d_4y_1 \bmod{4}$.
Moreover, the summation is now over
$\y\in \SS_*$. 
In line with \eqref{eq:e-game}, and using the estimate \eqref{queue} to
handle large values of $\xi_4$, we therefore deduce from Lemma
\ref{lem:linear} and Hypothesis-$(*,0)$ that 
\begin{align*}
S(X)&=
\sum_{\xi_4=1}^\infty \frac{ \delta_{*}C_0 }{2^{\xi_4+2}} X^2
+O\Big(\frac{D^\ve  L_\infty^{ \varepsilon} r_\infty r' X^2}{(\log
X)^{\eta- \varepsilon}}\Big)
=C_0 X^2 
+O\Big(\frac{D^\ve  L_\infty^{ \varepsilon} r_\infty r' X^2}{(\log
X)^{\eta- \varepsilon}}\Big). 
\end{align*}
Thus $\del_{1,1}(\A)=1$ when $a_3'\equiv d_3 -2^{\mu_3}\bmod 4$, 
as claimed in \eqref{delta4}.

\subsection{The case $\nu_3=\nu_4=0$ and $\max\{\mu_3,\mu_4\}\geq 1>
\min\{\mu_3,\mu_4\}=0$}

The treatment of this case runs parallel to the previous section. For 
reasons of symmetry we may restrict ourselves to the case 
$\mu_3\geq 1$ and $\mu_4=0$. For $\x \in \SS_1\cap \mcal{X}$ the term
$2^{-k_3}L_3(\x)$ is odd, whereas $2^{-k_4}L_4(\x)$  is always
even. We now observe that $r(L_3(\x)/d_3)$  is non-zero if and only if
$x_2\equiv b_3'd_3-2^{\mu_3} \bmod 4$.  Our task is to show that
\eqref{delta5} holds.

Let us write $\xi_4=\nu_2(2^{-k_4}L_4(\x))\geq 1$. Arguing as above we see that in order for 
$r(2^{-k_4-\xi_4}L_4(\x)/d_4)$
to be non-zero, it is necessary and sufficient that \eqref{eq:1711.1}
holds. In particular we must take care to sum only over
those $\xi_4$ for which 
\begin{equation}
  \label{eq:1611.2}
a_4' +b_3'b_4'd_3\equiv 2^{\mu_3}+2^{\xi_4} \pmod{4}.  
\end{equation}
For each such ${\xi_4}$ 
we make the transformation \eqref{eq:0112.1} as above, which
again brings $L_3,L_4$ into a satisfactory shape for
\textsf{NH}$_0(\ma{d})$, 
and the summation is over $\y\in \SS_*$. 
We may now deduce from Lemma
\ref{lem:linear} and Hypothesis-$(*,0)$, together with the argument
involving \eqref{queue}, that 
$$
S(X)=
\sum_{\xi_4} \frac{ \delta_{*}C_0 }{2^{\xi_4+2}} X^2
+O\Big(\frac{D^\ve  L_\infty^{ \varepsilon} r_\infty r' X^2}{(\log
X)^{\eta- \varepsilon}}\Big), 
$$
where the sum is over $\xi_4\geq 1$ such that \eqref{eq:1611.2}
holds. If $a_4' +b_3'b_4'd_3-2^{\mu_3} \equiv 2 \bmod{4}$, 
then we must restrict attention to the single value $\xi_4=1$, which gives
$\del_{1,1}(\A)=1/2$. If however $a_4' +b_3'b_4'd_3-2^{\mu_3} \equiv 0
\bmod{4}$,  
then we must restrict attention to $\xi_4\geq 2$, giving 
$\del_{1,1}(\A)=\sum_{\xi_4=2}^\infty 2^{-\xi_4}=1/2.$
This therefore confirms \eqref{delta5}.

\subsection{The case $\mu_3=\nu_3=\mu_4=\nu_4=0$}

We reason in an analogous manner to the previous sections. 
Our valuation of $\delta_{1,1}(\A)$ will depend on the $2$-adic
valuation $v$ of $a'_3b'_4-a'_4b'_3$, as defined in \eqref{eq:v}. 
Our aim is to show that \eqref{delta6} holds.

Let  $\x \in \SS_1\cap \mcal{X}$, and introduce parameters $\xi_3,\xi_4\geq
1$ such that \eqref{eq:extract2} holds for $j=3,4$. 
Let us deal with the case $\xi_4\geq \xi_3$.
The system 
$$
a'_3x_1+b_3'x_2\equiv 0 \pmod{2^{\xi_3}},\quad a'_4x_1+b_4'x_2\equiv 0
\pmod{2^{\xi_4}}
$$ 
is equivalent to 
$$
(a'_3b'_4-a'_4b'_3)x_1\equiv 0 \pmod{2^{\xi_3}},\quad a'_4x_1+b_4'x_2 \equiv 0
\pmod{2^{\xi_4}}.
$$
Let us write $a'_3b'_4-a'_4b'_3=2^{v}c_{34}$, with $c_{34}$ odd.  We
clearly have $\xi_3\leq v$. Moreover, the term $r( 2^{-k_4-\xi_4}L_4(\x)/d_4)$ is
non-zero if and only if \eqref{eq:1711.1}
holds.  Assuming this to be the case, we must therefore have
$$ 
a'_3x_1+b_3'x_2\equiv 
\big(a_3'+b_3'\overline{b_4'}( d_42^{\xi_4}-a_4')\big)x_1
\equiv 
\overline{b_4'}c_{34}2^v+ b_3'\overline{b_4'} d_4 2^{\xi_4}
\pmod{2^{\xi_3+2}}.
$$
Provided that
\begin{equation}
  \label{eq:1711.2}
 \overline{b_4'}c_{34}2^v+ b_3'\overline{b_4'} d_4 2^{\xi_4} 
\equiv 2^{\xi_3} d_3\pmod{2^{\xi_3+2}},   
\end{equation}
therefore, it follows that we may again carry out the transformation 
\eqref{eq:0112.1} to bring  $L_3,L_4$ into a satisfactory shape for
\textsf{NH}$_0(\ma{d})$. The summation is now over $\y\in \SS_*$.
We easily deduce from Lemma \ref{lem:linear} and Hypothesis-$(*,0)$
that there is the contribution 
$$
\frac{ \delta_{*}C_0 }{2^{\xi_4+2}} X^2
+ O\Big(\frac{D^\ve  L_\infty^{ \varepsilon} r_\infty r' 2^{\varepsilon
    \xi_4} X^2}{(\log X)^{\eta- \varepsilon}}\Big),
$$
for fixed $1\leq \xi_3\leq \xi_4$ such that \eqref{eq:1711.2}
holds.  
Using an estimate of the type \eqref{queue}, it is an easy
matter to deduce that the overall contribution to the error in summing
over the available $\xi_3,\xi_4$ is 
$ O\big( {D^\ve  L_\infty^{ \varepsilon} r_\infty r' X^2}{(\log
X)^{-\eta +\varepsilon}}\big)$. 
Moreover, we deduce that
$$
\del_{1,1}(\A)=
\sum_{\xi_3=\xi_4} \frac{1}{2^{\xi_4}}+2
\sum_{\xi_3<\xi_4} \frac{1}{2^{\xi_4}}, 
$$
for a summation over $\xi_3,\xi_4\geq 1$ such that \eqref{eq:1711.2} holds. To evaluate
this quantity we  consider a number of subcases, beginning
with the contribution from $\xi_3=\xi_4$.  Then we must have
$1\leq \xi_3\leq v-1$ and $b_3'\overline{b_4'}d_4+2^{v-\xi_3}\equiv d_3 \bmod
4$. Let us write $W_1$ for the set of all such positive integers $\xi_3$.
Then we obtain the overall contribution
\begin{equation}
  \label{eq:1711.3}
\sum_{\xi\in W_1}
\frac{1}{2^{\xi}}=
\left\{
\begin{array}{ll}
0, &\mbox{if $v=1$},\\
1-1/2^{v-2}, &\mbox{if $v\geq 2$ and $b_3'd_3\equiv b_4'd_4 \bmod{4}$},\\
1/2^{v-1}, &\mbox{if $v\geq 2$ and $b_3'd_3\equiv -b_4'd_4    \bmod{4}$},\\ 
\end{array}
\right.
\end{equation}
Turning to the contribution from $\xi_3<\xi_4$, it follows from
\eqref{eq:1711.2} that $\xi_3=v$ and
$ \overline{b_4'}c_{34} + 2^{\xi_4-v}\equiv d_3\ \bmod 4$. 
Write $W_2$ for the set of all such vectors $(\xi_3,\xi_4)\in \N^2$.
Then a little thought reveals that we obtain a contribution
$$
2\sum_{(\xi_3,\xi_4)\in W_2}
\frac{1}{2^{\xi_4}}=\frac{1}{2^{v}}
$$
from this case.  Combining this with \eqref{eq:1711.3}, we therefore
conclude the proof of \eqref{delta6}.

\end{document}